\documentclass[a4paper,reqno]{amsart}
\usepackage[english]{babel}
\usepackage{amssymb, amsmath, amsthm,mathrsfs} 
\usepackage{bbm,amsmath,amsthm,epsfig,latexsym,marvosym, esint}
\usepackage{amsfonts}
\usepackage{bigints}
\usepackage{amsmath}
\usepackage{bbm}
\usepackage{amssymb}
\usepackage{enumerate}
\usepackage{ esint }
\usepackage{hyperref}
\usepackage{color}
\usepackage[font={small,up}]{caption}

\numberwithin{equation}{section} 

\usepackage{color}
\usepackage{hyperref}
\definecolor{ddorange}{rgb}{1,0.5,0}
\definecolor{ddcyan}{rgb}{0,0.2,1.0}

\def\Xint#1{\mathchoice
{\XXint\displaystyle\textstyle{#1}}%
{\XXint\textstyle\scriptstyle{#1}}%
{\XXint\scriptstyle\scriptscriptstyle{#1}}%
{\XXint\scriptscriptstyle\scriptscriptstyle{#1}}%
\!\int}
\def\XXint#1#2#3{{\setbox0=\hbox{$#1{#2#3}{\int}$}
\vcenter{\hbox{$#2#3$}}\kern-.5\wd0}}

\def\dashint{\Xint-}

\usepackage{amsthm}
\makeatletter
\def\th@plain{%
  \thm@notefont{}
  \itshape 
}
\def\th@definition{%
  \thm@notefont{}
  \normalfont 
}
\makeatother
\newtheorem{thm}{Theorem}[section]
\newtheorem{prop}[thm]{Proposition}
\newtheorem{lem}[thm]{Lemma}

\theoremstyle{definition}
\newtheorem{defn}[thm]{Definition}

\newtheorem{oss}[thm]{Remark}

\newcommand{\N}{\mathbb{N}}
\newcommand{\R}{\mathbb{R}}

\renewcommand{\epsilon}{\varepsilon}

\newcommand{\sm}{\setminus}

\newcommand{\hd}{\mathcal{H}^{d-1}}
\newcommand{\Rd}{\mathbb{R}^d}

\newcommand{\Mdd}{{\R^{d\times d}_{sym}}}

\newcommand{\e}{\varepsilon}

\title[Non-local approximation of the Griffith functional]{Non-local approximation of the Griffith \\ functional}
\author[Giovanni Scilla]{Giovanni Scilla}
\address[Giovanni Scilla]{Dipartimento di Matematica ed Applicazioni ``R. Caccioppoli'', Universit\`{a} di Napoli Federico~II, Via Cintia Monte Sant'Angelo, 80126 Napoli, Italy}
\email[Giovanni Scilla]{giovanni.scilla@unina.it}
\author[Francesco Solombrino]{Francesco Solombrino}
\address[Francesco Solombrino]{Dipartimento di Matematica ed Applicazioni ``R. Caccioppoli'', Universit\`{a} di Napoli Federico~II, Via Cintia Monte Sant'Angelo, 80126 Napoli, Italy}
\email[Francesco Solombrino]{francesco.solombrino@unina.it}

\begin{document}
\begin{abstract} An approximation, in the sense of $\Gamma$\hbox{-}convergence and in any dimension $d\geq1$, of Griffith-type functionals, with $p-$growth ($p>1$) in the symmetrized gradient, is provided by means of a sequence of non-local integral functionals depending on the average of the symmetrized gradients on small balls. 
\end{abstract}

\maketitle

\setcounter{tocdepth}{1}  
\tableofcontents


\bigskip
\bigskip

\section{Introduction}

In this paper we provide a variational approximation by means of {\it non-local integral energies} of functionals of the form
\begin{equation}
\alpha\int_{\Omega\backslash K} W(\mathcal{E}u(x))\,\mathrm{d}x  + 2\beta \mathcal{H}^{d-1}(K), 
\label{Griff1}
\end{equation}
where $\Omega$ is a bounded {open} subset of $\R^d$, $K\subseteq\Omega$ is closed, $W$ is a non-negative and convex function with $p-$growth for some $p>1$, $u\in C^1(\Omega\backslash K;\R^d)$, $\mathcal{E}u$ denotes the symmetric part of the gradient of $u$ and $\mathcal{H}^{d-1}$ is the $(d-1)$-dimensional Hausdorff measure. 

Functionals as in \eqref{Griff1} are the core of many variational models of fracture mechanics, in the framework of Griffith's theory of brittle fracture under the small strain assumption (see, e.g., \cite{Griffith, FM} and the references in \cite[Introduction]{CC}). If, as usual, the set $\Omega$ denotes the reference configuration and $u$ represents the displacement field of the body, then the total energy \eqref{Griff1} is the sum of a bulk energy in $\Omega\backslash K$, where the material is supposed to have an elastic/elasto-plastic behavior (see, e.g., \cite[Section 2]{FM}, \cite[Sections 10 and 11]{Hu}), and a surface term accounting for the energy necessary to produce the fracture, proportional to the area of the crack surface $K$. It is only very recently that a rigorous weak formulation of the problem \eqref{Griff1} has been provided \cite{DM2013,CC18Comp}. Within this setting, $u$ is a (vector-valued) generalized special function of bounded deformation, for which the symmetrized gradient $\mathcal{E}u$ is defined almost everywhere in an approximate sense (see~\cite{DM2013}), and the set $K$ is replaced by the $(d-1)$-rectifiable set $J_u$, the jump set of $u$. This space is denoted by $GSBD^p(\Omega)$, where the exponent $p$ refers to the integrability of $\mathcal{E}u$. After that the existence of weak minimizers has been achieved, one can actually show that the jump set thereof is closed (up to a  $\mathcal{H}^{d-1}$-negligible set), and prove well-posedness of the minimization problem for \eqref{Griff1} (see \cite{iur1, iur2, vito}).

However, the minimization of functionals of the type \eqref{Griff1} may be a hard task in practice, mainly due to the presence of the surface term $\mathcal{H}^{d-1}(J_u)$. Such difficulties already appear for $W(M)=|M|^2$ and in the case of antiplane shear (see, e.g., \cite{BFM}), where the energy \eqref{Griff1} reduces to the
Mumford-Shah-type functional
\begin{equation}
\int_\Omega |\nabla u|^2\, \mathrm{d}x + \mathcal{H}^{d-1}(J_u), 
\label{MS}
\end{equation}
for a scalar-valued displacement $u\in SBV(\Omega)$, the space of special functions of bounded variation. It is indeed well-known that a variational approximation of \eqref{MS} by means of local integral functionals of the form
\begin{equation*}
\int_\Omega f_\varepsilon(\nabla u)\,\mathrm{d}x\,,
\end{equation*}
defined on Sobolev spaces, is doomed to failure (see \cite[Introduction]{BraDal97}).
Over the last years,  this has motivated a great effort to provide suitable approximations of \eqref{MS} by means of more manageable functionals, leading to the convergence of minimum points. A number of different approaches has been proposed, which in some cases have also been generalised to the (more challenging) setting of \eqref{Griff1}, both in a discrete and in  a continuous, infinite-dimensional, setting.

A very popular approach, originally proposed by Ambrosio and Tortorelli ~\cite{AT1,AT2} and generalised to the linearly elastic setting in \cite{Foc-Iur, CC}, provides an approximation of \eqref{Griff1} and \eqref{MS} by elliptic functionals (with parameter $\epsilon>0$) at the price of adding an auxiliary {\it phase-field} variable $v \in H^1 (\R^d; [0,1])$. The heuristics behind these functionals, taking the form
\[
\int_{\Omega} v_\e(x)\,W(\mathcal{E}u_\e(x))\,\mathrm{d}x  + \frac{1}\e \int_\Omega (v_\e(x)-1)^2\,\mathrm{d}x+ \e \int_\Omega |\nabla v_\e(x)|^2\,\mathrm{d}x
\]
is to approximate the discontinuity set $K$ with the $\e$-layer $\{v_\e \sim 0\}$. Also discretizations of the above functionals by means of either finite-difference or finite-elements with mesh-size $\delta$, independent of $\epsilon$, have been considered. For a suitable fine mesh, with size $\delta=\delta(\epsilon)$ small enough, these numerical approximations $\Gamma$-converge, as $\epsilon\to0$, to the Mumford-Shah functional (see \cite{BBZ, BelCos94}, and \cite{BCR} for the case of a stochastic lattice). A similar result for the energy \eqref{Griff1} has been recently provided in \cite{CSS}. For other discrete approaches  based on finite differences or finite elements we  may mention \cite{Chambolle1, ChambDM, Strekalovskiy2014127}, in the context of the Mumford-Shah functional, and \cite{Neg03, AFG} for the Griffith model.

Closer to the purpose of our paper are, however,  variational approximations by means of {\it nonlocal} integral energies. Following a conjecture by De Giorgi, Gobbino proved for instance in \cite{Gob98} that the functionals
\[
\frac{1}{\e^{d+1}}\int_{\R^d\times \R^d}\arctan\left(\frac{|u(x)-u(y)|^2}{|y-x|}\right)\mathrm{e}^{-\frac{|y-x|^2}{\e^2}}\,\mathrm{d}x\,\mathrm{d}y
\] 
$\Gamma$-converge to \eqref{MS} when $\e \to 0$. A discretization of this model on graphs has been recently analysed in \cite{Caroccia_2020}, and adaptions to the stochastic setting have also been provided (\cite{Ruf}).
Another method, introduced in \cite{BraDal97}, is based on non-local integral functionals whose density depends on the average of the gradient on small balls, {in order to prevent large gradients being concentrated in small regions.} There, functionals of the form
\begin{equation}
F_\epsilon(u):=\int_\Omega f\left(\epsilon\dashint_{B_\epsilon(x)\cap\Omega} |\nabla u(y)|^2\,\mathrm{d}y\right)\,\mathrm{d}x
\label{eq:bradal97}
\end{equation}
are considered, where $f:[0,+\infty)\to[0,+\infty)$ is an increasing function such that
\begin{equation}\label{eq: f-intro}
\lim_{t\to0^+}\frac{f(t)}{t}=\alpha\,,\quad \lim_{t\to+\infty} f(t)=\beta\,,
\end{equation}
$B_\epsilon(x)$ denotes the {open} ball of radius $\epsilon$ centred at $x\in\Omega$ and $\dashint_B v\,\mathrm{d}x$ is the average of $v$ on $B$. The functionals \eqref{eq:bradal97} $\Gamma$\hbox{-}converge, as $\epsilon\to0$, to the functional
\begin{equation*}
F(u):= \alpha \int_\Omega |\nabla u(x)|^2\,\mathrm{d}x + 2\beta \mathcal{H}^{d-1}(J_u)\,.
\end{equation*}
{Afterwards}, in \cite{BrG}, it has been shown that more general energies of the form
\begin{equation*}
\int_\Omega |\nabla u(x)|^2\,\mathrm{d}x + \int_{J_u}\theta(|u^+-u^-|)\,\mathrm{d}\mathcal{H}^{d-1}\,,
\end{equation*}
where $|u^+-u^-|$ is the jump of $u$ across $J_u$, can be obtained by considering non-local approximating functionals as in \eqref{eq:bradal97} with varying densities $f=f_\varepsilon$, and the function $\theta$ is computable from $f_\varepsilon$. This analysis has been continued in \cite{LV1,LV} for functionals with bulk terms having linear growth in the gradient.

The very first non-local approximation of Griffith-type energies on the footsteps of \cite{BraDal97}, inspired by the subsequent generalization \cite{CT} of such model, has been provided in \cite{Negri2006}. There, non-local convolution-type energies of the form
\begin{equation}
\int_\Omega f_\epsilon\left(\int_{\epsilon \, {\rm supp}\rho} \left|\mathcal Eu(y)\right|^p\rho_\epsilon(x-y)\,\mathrm{d}y\right)\,\mathrm{d}x
\label{eq:Negri}
\end{equation}
are considered, where $f_\epsilon$ is a suitable sequence of densities, $\rho$ is a convolution kernel with support ${\rm supp}\rho$ and $\rho_\epsilon(z)$ is the usual sequence
of convolution kernels $\rho(z/\epsilon)/\epsilon^d$. The $\Gamma$\hbox{-}limit of \eqref{eq:Negri} with respect to the $L^1$ convergence is shown to be the Griffith-type functional
\begin{equation*}
\int_{\Omega}|\mathcal{E}u(x)|^p\,\mathrm{d}x  + \int_{J_u} \phi_\rho(\nu)\,\mathrm{d}\mathcal{H}^{d-1}\,, 
\end{equation*}
where the anisotropy $\phi_\rho$ depends on the geometry and on the size of ${\rm supp}\rho$, and the function $u$ belongs to the space $SBD^p$ of special functions of bounded deformation with $\mathcal{E}u\in L^p$, which is a (proper) subspace of $GSBD^p$. The argument in \cite{Negri2006} introduces some novelties with respect to \cite{BraDal97, CT}, in particular for the proof of the lower bound, which is obtained by means of a delicate construction based on a slicing technique. However, as it happens when dealing with the space $SBD^p$, in order to obtain compactness of sequences of competitors with equibounded energy, an $L^\infty$ bound has to be imposed, which is quite unnatural in Fracture Mechanics.

\emph{Our results:} 
The purpose of our paper is to provide a variational approximation of the functional \eqref{Griff1} in the spirit of \cite{BraDal97}. We will namely show that for $f$ complying with \eqref{eq: f-intro}, the functionals
\begin{equation}
F_\epsilon(u):=\int_\Omega f\left(\epsilon\dashint_{B_\epsilon(x)\cap\Omega} W\left(\mathcal Eu(y)\right)\,\mathrm{d}y\right)\,\mathrm{d}x
\label{eq:ours}
\end{equation}
$\Gamma$-converge to the functional \eqref{Griff1} in the $L^1(\Omega)$-topology (Theorem \ref{thm:mainresult}). The proof strategy is based on the localization method for $\Gamma$-convergence (see, e.g., \cite[Chapters 14--20]{DM93}). One first considers, for any open subset $A\subset \Omega$, the localized functionals $F_\e(\cdot, A)$ defined as in \eqref{eq:ours} with $A$ in place of $\Omega$ and their asymptotic behavior. The core of the argument (essentially contained in Propositions \ref{prop:estimate}, \ref{prop:lowboundjump}, and \ref{prop:boundsp}) consists in showing that the lower $\Gamma$-limit $F'(u, A)$ satisfies the estimates
\begin{equation}
F'(u, A)\geq \alpha \int_A W(\mathcal{E}u)\,\mathrm{d}x\,,\quad F'(u, A)\geq 2\beta \int_{A\cap J_u} |\langle \nu_u, \xi \rangle|\,\mathrm{d}\mathcal{H}^{d-1}
\label{eq: core}
\end{equation}
for each $u \in GSBD^p(A)$, $A\subset \Omega$,  and $\xi$ unit vector in $\R^d$. Above, the symbol $\langle\cdot,\cdot\rangle$ denotes the scalar product in $\mathbb{R}^d$. As the two terms on the right-hand side are mutually singular, the $\Gamma$-liminf estimate can be {obtained from} these two separate estimates by a standard technique (Lemma \ref{lem: lemmasup}). 

While this general scheme has also been pursued in \cite{BraDal97}, getting to \eqref{eq: core} is rather different in our paper than it was to obtain analogous estimates in theirs. Indeed, in the $SBV$-context of Mumford-Shah-type functionals, one has the possibility of lowering the energy by truncating competitors. Hence, the main estimates can be proved for functions in $SBV\cap L^\infty$, as it is done for instance in \cite[Proposition 4.1 and Proposition 5.1]{BraDal97}, where $L^\infty$-bounds are explicitly exploited. A similar tool is not available in the bounded deformation setting. Hence, we have to renounce the semi-discrete approach of \cite{BraDal97} and follow a different strategy, which is closely related to the the heuristics of the model \eqref{eq:ours}. \\
\indent
The main idea for obtaining the first estimate in \eqref{eq: core} is contained in the proof of Proposition \ref{prop:compactness}. Using an energy estimate and the coarea formula, there we show that, for a given error parameter $\delta$, the set where  the averages $\e\dashint_{B_{(1-\delta)\e}(x)}W(\mathcal{E}u(x))\,\mathrm{d}x$ exceed a given threshold can be included in a set $K'_\e$ with vanishing area and bounded perimeter. This allows one to show that the $\Gamma$-limit of the energies \eqref{eq:ours} is controlled from below by a functional of the type \eqref{Griff1}. The optimal constant in the bulk term can be recovered, as done in Proposition \ref{prop:estimate}, by replacing a sequence of competitors $(u_\e)$ with their averages on balls of radius $\e$ at points $x \in \Omega \setminus K'_\e$. Indeed, $K'_\e$ is the set where, intuitively, the energy does  concentrate on lower dimensional manifolds and the bulk contribution can be neglected. It is worth mentioning that such an optimal estimate for the bulk term is not derived by means of any slicing procedure, which would not comply well with the general form of the bulk energy we are considering.\\
\indent
The second estimate in \eqref{eq: core} is instead obtained by means of a slicing argument in the fixed direction $\xi$ (see Proposition \ref{prop:lowboundjump}), first reconducting the problem to the analysis of the one-dimensional version of the functional \eqref{eq:bradal97} (which can be performed with elementary arguments, see \cite[Theorem~3.30]{Braides1}) and then exploiting the slicing properties of $GSBD$ functions recalled in Section \ref{sec: gsbd}. Finally, the $\Gamma$-limsup inequality (Proposition \ref{prop:upperbound}) can be obtained by a direct construction for a regular class of competitors having a ``nice'' jump set, and which are dense in energy according to recent approximation results {by Chambolle and Crismale \cite{CC}, summarized in Theorem \ref{thm:density}, and by Cortesani and Toader \cite{CT}.} 
Let us remark that our proof strategy can also be applied, with obvious modifications, for an alternative and,  in our opinion, slightly simpler proof of the results in \cite{BraDal97}.\\
\indent
To end up this review of our results, we want to motivate our choice of  the $L^1$-topology and warn the reader of a related issue. Actually, while $L^1$-convergence is a natural choice in the context of the Mumford-Shah functional, both for the possibility of using truncations and for the presence of $L^p$-fidelity terms, when dealing with fracture models it would be preferable to deal with the convergence in measure. Indeed, Proposition \ref{prop:compactness} in principle only allows one for applying Theorem \ref{th: GSDBcompactness}, which provides subsequences that are (essentially) converging in measure. \footnote{The presence of the exceptional set $A^\infty$ in the statement of Theorem \ref{th: GSDBcompactness} is no real issue in the context of the Griffith model, as setting $u=0$ there is optimal for the energy, see \cite{CC18Comp}.} However, dealing with sequences converging in $L^1$ allows us to deduce the convergence of the averaged functions in Lemma \ref{lem:lbp} which are a useful tool in our proofs. Notice that compactness in $L^1$ can be easily enforced by adding a lower-order fidelity term, for which a completely satisfactory compactness and $\Gamma$-convergence result can be stated and proved (Theorem \ref{thm:mainresult2}). It then seems to us that adding such a term, although not completely justified from the point of view of fracture mechanics, does not really affect our methods and results. \\ 
\indent
\emph{Outline of the paper:}  The paper is organized as follows. In Section~\ref{sec:notation} we fix the basic notation and collect some definitions and results on the function spaces we will deal with (Section~\ref{sec: gsbd}), together with some technical lemmas (Section~\ref{sec:lemmas}) which will be useful throughout the paper. In Section~\ref{sec:model} we list the main assumptions, introduce our model (eq. \eqref{energies0}), and state the main results of the paper, given {in} Theorem~\ref{thm:mainresult} and Theorem~\ref{thm:mainresult2}. Section~\ref{sec:compactness} contains the compactness result of Proposition~\ref{prop:compactness}. Section~\ref{sec:estimbelow} is devoted to the $\Gamma$\hbox{-}liminf inequality: the separate estimates from below of the bulk term and the surface term of the energy are contained in Sections~\ref{sec:estimbelowbulk} and \ref{sec:estimbelowsurf}, respectively; the proof of the $\Gamma$\hbox{-}liminf inequality is the content of Section~\ref{sec:gammaliminf}. The upper bound is provided in Section~\ref{sec:upperbound}.

\section{Notation and preliminary results} \label{sec:notation}

\subsection{Notation}

The symbol $\langle\cdot,\cdot\rangle$ denotes the scalar product in $\mathbb{R}^d$, while $|\cdot|$ stands for the Euclidean norm in any dimension. The symbol $\Omega$ will always denote an open, bounded subset of $\mathbb{R}^d$ {with Lipschitz boundary}. The Lebesgue measure in $\mathbb{R}^d$ and the $s$-dimensional Hausdorff measure are written as $\mathcal{L}^d$ and $\mathcal{H}^s$, respectively.  

The symbol $S^{d-1}$ will denote the $(d-1)$-dimensional unit sphere. The family of the open subsets of $\Omega$ will be denoted by $\mathcal{A}(\Omega)$.


\subsection{$GBD$ and $GSBD$ functions}\label{sec: gsbd}

We recall here some basic definitions and results on generalized functions with bounded deformation, as introduced in \cite{DM2013}. Throughout the paper we will use standard notations for the spaces $(G)SBV$ and $(G)SBD$, referring the reader to \cite{AFP} and \cite{ACDM, BCDM, Temam}, respectively, for a detailed treatment on the topics.\\

Let $\xi\in\R^d\backslash\{0\}$ and $\Pi^\xi=\{y\in\R^d:\, \langle\xi,y\rangle=0\}$. If $y\in\Pi^\xi$ and $\Omega\subset\R^d$ we set $\Omega_{\xi,y}:=\{t\in\R:\, y+t\xi\in \Omega\}$ and $\Omega_\xi:=\{y\in \Pi^\xi:\, \Omega_{\xi,y}\neq\emptyset\}$. Given $u:\Omega\to\R^d$, $d\geq2$, we define $u^{\xi,y}: \Omega_{\xi,y}\to\R$ by 
\begin{equation}
u^{\xi,y}(t):=\langle u(y+t\xi),\xi\rangle\,, 
\label{section1}
\end{equation}
while if $h: \Omega\to\R$, the symbol $h^{\xi,y}$ will denote the restriction of $h$ to the set $\Omega_{\xi,y}$; namely,
\begin{equation}
h^{\xi,y}(t):= h(y+t\xi)\,.
\label{section2}
\end{equation}

Let $\xi\in S^{d-1}$. For any $y\in\R^d$ we denote by $y_\xi$ and $y_{\xi^\perp}$ the projections onto the subspaces $\Xi:=\{t\xi:\,\,t\in\R\}$ and $\Pi^\xi$, respectively. For $\sigma\in(0,1)$ and $x\in\R^d$ we define the cylinders
\begin{equation*}
C_\sigma^\xi(0):=\{y\in\R^d:\,\, |y_\xi|<\sigma\,,\,\, |y_{\xi^\perp}|<\sqrt{1-\sigma^2}\}\,,\quad C_\sigma^\xi(x):=x+C_\sigma^\xi(0)\,.
\end{equation*}
Note that it holds $C_\sigma^\xi(x)\subseteq B_1(x)$, and that  $C_\sigma^\xi(x)=(x_\xi-\sigma,x_\xi+\sigma)\times B^{d-1}_{\sqrt{1-\sigma^2}}(x_{\xi^\perp})$, where $B^{d-1}$ denotes a ball in the $(d-1)$-dimensional space $\Pi^\xi$. 

\begin{defn}
An $\mathcal L^{d}$-measurable function $u:\Omega\to \R^{d}$ belongs to $GBD(\Omega)$ if there exists a positive bounded Radon measure $\lambda_u$ such that, for all $\tau \in C^{1}(\R^{d})$ with $-\frac12 \le \tau \le \frac12$ and $0\le \tau'\le 1$, and all $\xi \in S^{d-1}$, the distributional derivative $D_\xi (\tau(\langle u,\xi\rangle))$ is a bounded Radon measure on $\Omega$ whose total variation satisfies
$$
\left|D_\xi (\tau(\langle u,\xi\rangle))\right|(B)\le \lambda_u(B)
$$
for every Borel subset $B$ of $\Omega$. 
\end{defn}

If $u\in GBD(\Omega)$ and $\xi\in\R^d\backslash\{0\}$ then, in view of \cite[Theorem~9.1, Theorem~8.1]{DM2013}, the following properties hold:
\begin{enumerate}
\item[{\rm(a)}] $\dot{u}^{\xi,y}(t)=\langle\mathcal{E}u(y+t\xi)\xi,\xi\rangle$ for a.e. $t\in \Omega_{\xi,y}$;\\
\item[{\rm(b)}] {$J_{u^{\xi,y}}=(J_u^\xi)_{\xi,y}$} for $\mathcal{H}^{n-1}$-a.e. $y\in\Pi^\xi$, where
\begin{equation}
J_u^\xi:=\{x\in J_u:\, \langle u^+(x)-u^-(x),\xi\rangle\neq0\}\,.
\label{jumpset}
\end{equation}
\end{enumerate}

\begin{defn}
A function $u \in GBD(\Omega)$ belongs to the subset $GSBD(\Omega)$ of special functions of bounded deformation if in addition for every $\xi \in S^{d-1}$ and $\mathcal H^{d-1}$-a.e.\ $y \in \Pi^\xi$, the function $u^{\xi,y}$ belongs to {$SBV_{\mathrm{loc}}(\Omega_{\xi,y})$}.
\end{defn}

By \cite[Remark 4.5]{DM2013} one has the inclusions $BD(\Omega)\subset GBD(\Omega)$ and $SBD(\Omega)\subset GSBD(\Omega)$, which are in general strict. Some relevant properties of functions with bounded deformation can be generalized to this weak setting: in particular, in \cite[Theorem 6.2 and Theorem 9.1]{DM2013} it is shown that the jump set $J_u$ of a $GBD$-function is $\mathcal H^{d-1}$-rectifiable and that $GBD$-functions have an approximate symmetric differential $\mathcal{E}u(x)$ at $\mathcal L^{d}$-a.e.\ $x\in \Omega$, respectively. {Let $p>1$.} The space $GSBD^p(\Omega)$ is defined through:
$$
GSBD^p (\Omega):= \{u \in GSBD(\Omega): \mathcal{E}u \in L^p (\Omega; \mathbb R_{\mathrm{sym}}^{d\times d})\,,\,\mathcal H^{d-1}(J_u) < +\infty\}\,.
$$

Every function in $GSBD^p(\Omega)$ is approximated by bounded $SBV$ functions with more regular jump set, as stated by the following result (\cite[Theorem~1.1]{CC}). 

\begin{thm}
Let $\Omega\subset\mathbb{R}^d$ be a bounded open Lipschitz set, and let $u\in GSBD^p(\Omega;\R^d)$. Then there exists a sequence {$(u_n)$} such that\\
${\rm (i)}$ $u_n\in SBV^p(\Omega;\R^d)\cap L^\infty(\Omega;\R^d)$;\\
${\rm (ii)}$ each $J_{u_n}$ is closed and included in a finite union of closed connected pieces of $C^1$-hypersurfaces;\\
${\rm (iii)}$ $u_n\in W^{1,\infty}({\Omega}\backslash J_{u_n};\R^d)$, and
\begin{align}
u_n\to u \mbox{ in measure on $\Omega$},\label{convme}\\
\mathcal{E}u_n \to \mathcal{E}u \mbox{ in $L^p(\Omega;\R^{d\times d}_{sym})$,}\label{convgrad}\\
\mathcal{H}^{d-1}(J_{u_n}\triangle J_u)\to0.\label{convjump}
\end{align}
Moreover, if $\int_\Omega \psi(|u|)\,\mathrm{d}x$ is finite for $\psi:[0,+\infty)\to[0,+\infty)$ continuous, {increasing}, with
\begin{equation*}
\psi(0)=0,\,\,\, \psi(s+t)\le C(\psi(s)+\psi(t)), \,\,\, \psi(s)\le C(1+s^p), \,\,\, {\lim_{s\to+\infty}{\psi(s)}=+\infty}
\end{equation*} 
then
\begin{equation}
\lim_{n\to+\infty}\int_\Omega \psi(|u_n-u|)\,\mathrm{d}x=0\,.
\label{eq:densityfidel}
\end{equation}
\label{thm:density}
\end{thm} 

A further approximation result, by Cortesani and Toader \cite[Theorem~3.9]{CorToa99},  
allows us to approximate $GSBD^p(\Omega)$ functions with the so-called ``piecewise smooth''  $SBV$-functions, denoted $\mathcal{W}(\Omega;\Rd)$, characterized by the three properties 
\begin{equation}\label{1412191008}
\begin{cases}
u\in SBV(\Omega;\Rd)\cap W^{m,\infty}(\Omega\sm J_u;\Rd) \,\text{for every }m\in \N\,,\\
\hd(\overline{J}_u \sm J_u ) = 0\,,\\
\overline{J}_u \text{ is the intersection of $\Omega$ with a finite union of ${(d{-}1)}$-dimensional simplexes}\,.
\end{cases}
\end{equation}
{Notice that for the results above only the coercivity of $\psi$ is needed, while we will require $\psi$ to be superlinear at infinity in order to infer strong $L^1$-convergence (see Remark~\ref{rem: compactness} below).}

We recall the following general $GSBD^p$ compactness result from \cite{CC18Comp}, which generalizes \cite[Theorem 11.3]{DM2013}.  In the statement the symbol $\partial^*$ denotes the essential boundary of a set with finite perimeter.  We keep this general form of the statement. However, since we have to enforce $L^1$-convergence of sequences with bounded energy, the situation which is relevant for our purposes is described in Remark \ref{rem: compactness} below.
\begin{thm}[$GSBD^p$ compactness]\label{th: GSDBcompactness}
 Let $\Omega \subset \R$ be an open, bounded set,  and let $(u_n)_n \subset  GSBD^p(\Omega)$ be a sequence satisfying
$$ \sup\nolimits_{n\in \N} \big( \Vert \mathcal{E}u_n \Vert_{L^p(\Omega)} + \mathcal{H}^{d-1}(J_{u_n})\big) < + \infty.$$
Then there exists a subsequence, still denoted by {$(u_n)$}, such that the set  ${A^\infty} := \lbrace x\in \Omega: \, |u_n(x)| \to +\infty \rbrace$ has finite perimeter, and  there exists  $u \in GSBD^p(\Omega)$ such that 
\begin{align}\label{eq: GSBD comp}
{\rm (i)} & \ \ u_n \to u \  \ \ \  \mbox{ in measure on $\Omega\sm A^\infty$}, \notag \\ 
{\rm (ii)} & \ \ \mathcal{E}u_n \rightharpoonup\mathcal{E}u \ \ \ \text{ in } L^p(\Omega \setminus A^\infty; \Mdd),\notag \\
{\rm (iii)} & \ \ \liminf_{n \to \infty} \mathcal{H}^{d-1}(J_{u_n}) \ge \mathcal{H}^{d-1}(J_u \cup  (\partial^*A^\infty \cap\Omega)  ).
\end{align}
\end{thm}

\begin{oss}\label{rem: compactness}
If one additionally has
\[
 \sup\nolimits_{n\in \N}\int_\Omega  \psi(|u_n|)\,\mathrm{d}x<+\infty
\]
for a continuous, positive, increasing function $\psi$ satisfying
\begin{equation*}
\lim_{s\to+\infty}\frac{\psi(s)}{s}=+\infty\,,
\end{equation*}
then $u \in L^1(\Omega)$, so that {$A^\infty=\emptyset$}. Furthermore, ${\rm (i)}$ holds with respect to the $L^1$-convergence in $\Omega$, by the Vitali dominated convergence theorem.
\end{oss}

\subsection{Some lemmas}\label{sec:lemmas}

We recall here the following property of commutability of the integral averages with the gradient operator for a Sobolev vector-valued function. The proof is based on standard arguments by test functions, so we omit the details.

\begin{prop}
Let $u\in W^{1,p}(\Omega;\R^d)$. Let $\Omega' \subset\subset \Omega$ and $0\le \eta \le \mathrm{dist}(\Omega', \partial \Omega)$. Then the average
\begin{equation*}
\varphi(x):=\dashint_{B_\eta(x)} u(y)\,\mathrm{d}y
\end{equation*}
belongs to $W^{1,p}(\Omega';\R^d)$. Moreover, it holds that
\begin{equation}
\nabla \varphi (x) = \dashint_{B_\eta(x)} \nabla u(y)\,\mathrm{d}y \quad \mbox{ a.e. on $\Omega'$.}
\label{commut}
\end{equation}
\end{prop}

We will make use of the following convergence properties of averaged functions. These are probably well-known, a short proof is however added for the reader's convenience.

\begin{lem}\label{lem:lbp}
Assume that $w_{\e}\rightarrow w$ in $L^1(\Omega;\R^d)$ and let $\eta_\varepsilon$ be any sequence with $\eta_\varepsilon \to 0$ when $\varepsilon \to 0$. Then the following holds:
\begin{enumerate}
\begin{item}[{\rm (i)}]
 the sequence
	\[
	\hat{w}_\e(x):=\dashint_{B_{\eta_\varepsilon}(x)} w_{\e}(y)\,\mathrm{d} y 
	\]
satisfies $\hat{w}_\e \to w$ in $L^1(\Omega;\R^d)$;
\end{item}
\begin{item}[{\rm (ii)}]
for all $\xi \in S^{d-1}$ and a.e. \ $y \in \Pi^{\xi}$, the sequence
\[
\hat{w}^{\xi, y}_\e(t):=\dashint_{B^{d-1}_{\eta_\varepsilon}(y)} w_{\e}(z+t\xi)\,\mathrm{d}z
\]
satisfies  $\hat{w}^{\xi, y}_\e \to w^{\xi, y}$ in $L^1(\Omega_{\xi, y} ;\R^d)$, where $w^{\xi, y}(t):=w(y+t\xi)$.
\end{item}
\end{enumerate}
\end{lem}
\begin{proof}
We extend the functions $w_\e$ {to} the value $0$ in $\R^d \setminus \Omega$. For (i), we observe that with the change of variable $y=x+\eta_\e z$ and Fubini's theorem one has
\[
\int_\Omega |\hat{w}_\e(x)-w_\e(x)|\,\mathrm{d}x\le \dashint_{B_1(0)}\left(\int_\Omega \left|w_\e(x+\eta_\e z)-w_\e(x)\right|\,\mathrm{d}x\right)\,\mathrm{d}z
\]
For any fixed $\theta >0$ one has now, for small $\e$
\[
\int_\Omega |w_\e(x+\eta_\e z)-w_\e(x)|\,\mathrm{d}x\le \theta
\]
uniformly with respect to $z\in B_1(0)$ by the Fr\'echet-Kolmogorov criterion. It follows that
\[
\lim_{\e \to 0} \int_\Omega |\hat{w}_\e(x)-w_\e(x)|\,\mathrm{d}x=0\,,
\]
which gives (i).

Concerning (ii), write $x=y+t\xi$ and set $w^{\xi, y}_\e(t):=w_{\e}(y+t\xi)$. It holds that  $w^{\xi, y}_\e \to w^{\xi, y}$ in $L^1(\Omega_{\xi, y} ;\R^d)$ for a.e. $y \in \Pi^{\xi}$. For $\hat{w}^{\xi}_\e(x):=\hat{w}^{\xi, y}_\e(t)$ we have now  with the change of variable $z=y+\eta_\e z^\prime$ and Fubini's theorem that
\[
\int_\Omega |\hat{w}^{\xi}_\e(x)-w_\e(x)|\,\mathrm{d}x\le \dashint_{B^{d-1}_1(0)}\left(\int_\Omega \left|w_\e(x+\eta_\e z^\prime)-w_\e(x)\right|\,\mathrm{d}x\right)\,\mathrm{d}z^{\prime}
\]
so that, arguing as before,
\[
\hat{w}^{\xi}_\e-w_\e \to 0 \mbox{ in }L^1(\Omega; \R^d)\,.
\]
Hence, (ii) follows from the analogous convergence properties of the slices $w_\e^{\xi, y}(t)$.
\end{proof}

We will make also use of the following property of finite coverings of a bounded set.

\begin{oss}\label{rem:innercovering}
Let $R>0$ and $E\subset\R^d$ be such that $E\subset B_R$. Then, for every $r>0$, there exists a finite subset $E'\subset E$ such that
\begin{equation*}
E\subset \bigcup_{x\in E'} B_r(x)
\end{equation*}
and $\#(E')$ depends only on $r, R$ and $d$. If, in addition, $\frac{R}{r}\leq\lambda$, then there exists a constant $C=C(\lambda,d)$ such that $\#(E')\leq C$.

For this, we note that the family of balls
\begin{equation*}
\mathcal{B}:=\left\{B_{\frac{r}{2}}(z):\,\, z\in\left(\frac{r}{2\sqrt{d}}\mathbb{Z}\right)^d\cap B_R\right\}
\end{equation*}
is a covering of $B_R$. Now, for every ball $B_{\frac{r}{2}}(z)$ in $\mathcal{B}$ such that $B_{\frac{r}{2}}(z)\cap E\neq\emptyset$, we choose  a point $x=x(z)\in B_{\frac{r}{2}}(z)\cap E$ and so doing we construct the set 
\begin{equation*}
E':=\left\{x=x(z):\,\, z\in\left(\frac{r}{2\sqrt{d}}\mathbb{Z}\right)^d\cap B_R\,,\,\, B_{\frac{r}{2}}(z)\cap E\neq\emptyset\right\}\,.
\end{equation*}
Then, the desired covering of $E$ is given by $\{B_r(x):\,\, x\in E'\}$. Note also that
\begin{equation}
\#(E')\leq \#\left(\left(\frac{r}{2\sqrt{d}}\mathbb{Z}\right)^d\cap B_R\right):=C(r,R,d)<+\infty\,.
\label{bound}
\end{equation}
From \eqref{bound} and a simple scaling argument, we infer that $C(r,R,d)$ is uniformly bounded when the ratio $\frac{R}{r}$ is bounded.
\end{oss}


The following result, dealing with the supremum of a family of measures, will be useful for the derivation of the $\Gamma$\hbox{-}liminf
inequality (see, e.g., \cite[Proposition 1.16]{Braides1}).

\begin{lem}\label{lem: lemmasup}
Let $\mu:\mathcal{A}(\Omega)\longrightarrow[0,+\infty)$ be a superadditive function on disjoint open sets, let $\lambda$ be a positive measure on $\Omega$ and let $\varphi_h: \Omega\longrightarrow[0,+\infty]$ be a countable family of Borel functions such that $\mu(A)\geq\int_A\varphi_h\,\mathrm{d}\lambda$ for every $A\in\mathcal{A}(\Omega)$. Then, setting $\varphi:=\sup_{h\in\N}\varphi_h$, it holds that
\begin{equation*}
\mu(A)\geq\int_C\varphi\,\mathrm{d}\lambda
\end{equation*}
for every $A\in\mathcal{A}(\Omega)$.
\end{lem}

We also remark the following approximation property from below for lower semicontinuous increasing functions with truncated affine functions. A proof is given for the reader's convenience.
\begin{lem}\label{lem: belowapprox}
Consider a lower semicontinuous increasing function $f:[0,+\infty)\to[0,+\infty)$ such that there exist $\alpha, \beta>0$ with 
\[
\lim_{t\to 0^+}\frac{f(t)}t=\alpha, \quad \lim_{t\to +\infty} f(t)=\beta\,.
\]
Then there exist two positive sequences $(a_i)_{i\in \mathbb{N}}$, $(b_i)_{i\in \mathbb{N}}$ with
\[
\sup_i a_i= \alpha, \quad \sup_i b_i=\beta 
\]
and $\min\{a_i t, b_i\}\le f(t)$ for all $i \in \mathbb{N}$ and $t \in \R$.
\end{lem}

\begin{proof}
For all $(h, k) \in \mathbb{N}^2$ set
\[
a_{hk}:=\min\left\{\frac{f(t)}t \,: \quad t \in \left[0,\frac kh\right]\right\}, \quad b_{hk}:=f\left(\frac kh\right)\,.
\]
Above the function $\frac{f(t)}t$ is extended by continuity with the value $\alpha$ for $t=0$.
We clearly have that $a_{hk}\le \alpha$ for all $(h, k) \in \mathbb{N}^2$; furthermore $\alpha$ is an accumulation point for the family $a_{hk}$, hence
$\alpha=\sup\{a_{hk}: (h, k) \in \mathbb{N}^2\}$. With the monotonicity of $f$ we have  $b_{hk}\le \beta$ for all $(h, k) \in \mathbb{N}^2$; furthermore $\beta$ is an accumulation point for the family $b_{hk}$, hence
$\beta=\sup\{b_{hk}: (h, k) \in \mathbb{N}^2\}$. 

By construction we have
\[
a_{hk} t \le f(t) \mbox{ for all } t \in \left[0,\frac kh\right],\quad b_{hk} \le f(t) \mbox{ for all } t \in \left[\frac kh, +\infty\right),
\]
so that $\min\{a_{hk} t, b_{hk}\}\le f(t)$ for all $(h,k) \in \mathbb{N}^2$ and $t \in \R$. It then simply suffices to consider an enumeration of $\mathbb{N}^2$ to conclude the proof.
\end{proof}

\subsection{{$\Gamma$\hbox{-}convergence }}

{We recall here the definition of $\Gamma$\hbox{-}convergence for families of functionals depending on a real parameter. According to \cite{BraidesG}, we treat $\Gamma$\hbox{-}limits of functionals $F_\varepsilon:X\to[-\infty,+\infty]$ as $\varepsilon\to0^+$. The definition is an extension of that given for sequences of functionals labelled by a discrete parameter (see, e.g., \cite{DM93}), as we require all the properties to hold for every positive sequence $(\varepsilon_j)$ converging to $0$.}

{For all $u\in X$, we define the \emph{lower $\Gamma$\hbox{-}limit of $(F_\varepsilon)$} as $\varepsilon\to0^+$ by
\begin{equation}
F'(u):=\inf \left\{\displaystyle\mathop{\lim\inf}_{j\to+\infty}F_{\varepsilon_j}(u_j):\,\, \varepsilon_j\to0^+\,,\,\, u_j\to u\right\}\,,
\label{eq:Gammaliminf}
\end{equation}
and the \emph{upper $\Gamma$\hbox{-}limit of $(F_\varepsilon)$} as $\varepsilon\to0^+$ by
\begin{equation}
F''(u):=\inf \left\{\displaystyle\mathop{\lim\sup}_{j\to+\infty}F_{\varepsilon_j}(u_j):\,\, \varepsilon_j\to0^+\,,\,\, u_j\to u\right\}\,.
\label{eq:limsupchar}
\end{equation}
We then say that $(F_\varepsilon)$ \emph{$\Gamma$\hbox{-}converges to $F:X\to[-\infty,+\infty]$ as $\varepsilon\to0^+$} iff 
\begin{equation*}
F(u)=F'(u)=F''(u)\,,\quad \mbox{for all $u\in X$.}
\end{equation*}}

\subsection{A one-dimensional $\Gamma$\hbox{-}convergence result}

We recall here a one-dimensional $\Gamma$\hbox{-}convergence result which will be useful in the sequel. In the statement below, functions in $L^1(I)$ with $I \subset \R$ are extended by $0$ outside $I$, so that the functionals $H_\varepsilon$ are well-defined (actually, the result is not affected by the considered extension).

\begin{thm}\label{thm:braides}
Let $p>1$, let $I$ be a bounded interval in $\R$  and consider a lower semicontinuous increasing function {$f:[0,+\infty)\to[0,+\infty)$} such that there exist $\alpha, \beta>0$ with 
\[
\lim_{t\to 0^+}\frac{f(t)}t=\alpha, \quad \lim_{t\to +\infty} f(t)=\beta\,.
\]
Let $H_\varepsilon:L^1(I)\to[0,+\infty]$ be defined by
\begin{equation*}
H_\varepsilon(u):=\frac{1}{\varepsilon}\int_I f\left(\frac{1}{2}\int_{x-\varepsilon}^{x+\varepsilon}|u'(y)|^p\,\mathrm{d}y\right)\,\mathrm{d}x\,,
\end{equation*}
where it is understood that
\begin{equation*}
 f\left(\frac{1}{2}\int_{x-\varepsilon}^{x+\varepsilon}|u'(y)|^p\,\mathrm{d}y\right)=\beta
\end{equation*}
if $u\not\in W^{1,p}(x-\varepsilon,x+\varepsilon)$. Then the {functionals} $(H_\varepsilon)$ $\Gamma$\hbox{-}converge as $\varepsilon\to0^+$ to the functional
\begin{equation*}
H(u):=
\begin{cases}
\displaystyle \alpha\int_I|u'|^p\,\mathrm{d}t + 2\beta \#(J_u)\,, & \mbox{ if }u\in SBV(I)\,,\\
+\infty\,, & \mbox{ otherwise }
\end{cases}
\end{equation*}
in $L^1(I)$.
\end{thm}

\begin{proof}
 See \cite[Theorem~3.30]{Braides1}.
\end{proof}


\section{The non-local model and main results}\label{sec:model}
In this section we list our assumptions and introduce the main results of the paper.
Let $\Omega\subset\R^d$ be an open set with Lipschitz boundary, let $1< p<+\infty$ and $f:[0,+\infty)\to[0,+\infty)$ a lower semicontinuous, increasing function satisfying
\begin{equation}\label{eq: alfabeta}
\lim_{t\to 0^+}\frac{f(t)}t=\alpha>0, \quad \lim_{t\to +\infty} f(t)=\beta>0\,.
\end{equation} 
Let $W:\R^{d\times d}\to\R$ be a convex positive function on the subspace $\mathbb{M}^{d\times d}_{sym}$ of symmetric matrices, such that
\begin{equation}\label{eq:coerciv}
 W({\bf 0})=0\,,\quad c|M|^p\le W(M)\leq C(1+|M|^p)\,.
\end{equation}

For every $\varepsilon>0$ we consider the functional $F_\varepsilon: L^1(\Omega;\R^d)\to [0,+\infty]$ defined as
\begin{equation}
F_\varepsilon(u)=
\begin{cases}
\displaystyle\frac{1}{\varepsilon}\int_\Omega f\left(\varepsilon \dashint_{B_\varepsilon(x)\cap\Omega}W(\mathcal{E}u(y))\,\mathrm{d}y\right)\,\mathrm{d}x, & \mbox{ if }u\in W^{1,p}(\Omega;\R^d)\,,\\
+\infty\,, & \mbox{ otherwise on $L^1(\Omega;\R^d)$, }
\end{cases}
\label{energies0}
\end{equation}
where
\begin{equation*}
\dashint_{B}w(y)\,\mathrm{d}y := \frac{1}{\mathcal{L}^d(B)}\int_B w(y)\,\mathrm{d}y
\end{equation*}
for every Borel set $B\subseteq\Omega$ and for every $w\in L^1(B)$. 

We will deal with a localized version of the energies \eqref{energies0}. Namely, for every $A\subseteq\Omega$ open set, we will denote by $F_\varepsilon(u,A)$ the same functional as in \eqref{energies0} with the set $A$ in place of $\Omega$. \\

The following theorem is the  first main result of this paper.

\begin{thm}\label{thm:mainresult}
Under assumptions \eqref{eq: alfabeta} and \eqref{eq:coerciv}, it holds that
\begin{enumerate}
\begin{item}[{\rm (i)}]
there exists a constant $c_0$ independent of $\varepsilon$ such that, for all $(u_\varepsilon)\subset L^p(\Omega;\R^d)$ satisfying
$
F_\varepsilon(u_\varepsilon)
\leq C
$ {for every $\varepsilon>0$},
one can find a sequence $\overline{u}_\varepsilon \in {GSBV^p(\Omega;\R^d)}$ with
\[
\begin{split}
&\overline{u}_\varepsilon-u_\varepsilon \to 0 \mbox{ in measure  on }\Omega \\
&F_\varepsilon(u_\varepsilon)\geq c_0 \left(\int_{\Omega} W(\mathcal{E}\overline{u}_\varepsilon)\,\mathrm{d}x + 2 \mathcal{H}^{d-1}(J_{\overline{u}_\varepsilon} \cap \Omega) \right)\,.
\end{split}
\]
\end{item}
\begin{item}[{\rm (ii)}]
The {functionals} $(F_\varepsilon)$ $\Gamma$\hbox{-}converge, as $\varepsilon\to0$, to the functional
\begin{equation}
F(u)=
\begin{cases}
\displaystyle\alpha\int_\Omega W(\mathcal{E}u)\,\mathrm{d}x + 2\beta\,\mathcal{H}^{d-1}(J_u)\,\,, & \mbox{ if }u\in GSBD^p(\Omega)\cap L^1(\Omega;\R^d)\,,\\
+\infty\,, & \mbox{ otherwise on $L^1(\Omega;\R^d)$, }
\end{cases}
\label{limitfunct}
\end{equation}
with respect to the $L^1$ convergence in $\Omega$.
\end{item}
\end{enumerate}
\end{thm}

Notice that there is a mismatch between part (i) and (ii) of the previous statement. Indeed, the compactness property in (i) does not entail the  $L^1$-convergence of a subsequence of $(u_\varepsilon)$. It only allows one to apply Theorem \ref{th: GSDBcompactness}, which has a weaker statement.
However, the $L^1$-convergence on the whole $\Omega$ can be easily enforced with the addition of a lower order fidelity term, as we have discussed in Remark \ref{rem: compactness}. This motivates the statement below.

There, we consider a continuous increasing function $\psi:[0,+\infty)\to[0,+\infty)$ such that
\begin{equation}\label{eq: hpsi}
\psi(0)=0,\quad \psi(s+t)\le C(\psi(s)+\psi(t)), \quad \psi(s)\le C(1+s^p), \quad \lim_{s\to+\infty}\frac{\psi(s)}{s}=+\infty
\end{equation} 
and set for every open set $A\subset\Omega$
\begin{equation}
G_\varepsilon(u, A)=
\begin{cases}
\displaystyle F_\varepsilon(u, A)+\int_A\psi(|u|)\,\mathrm{d}x, & \mbox{ if }u\in W^{1,p}(A;\R^d)\,,\\
+\infty\,, & \mbox{ otherwise on $L^1(A;\R^d)$. }
\end{cases}
\label{energies1}
\end{equation}
When $A=\Omega$, we simply write $G_\varepsilon(u)$ in place of $G_\varepsilon(u, \Omega)$. Then we have the following result.

\begin{thm}\label{thm:mainresult2}
Under assumptions \eqref{eq: alfabeta}, \eqref{eq:coerciv},  and \eqref{eq: hpsi} it holds that
\begin{enumerate}
\begin{item}[{\rm (i)}]
If $(u_\varepsilon)\subset L^p(\Omega;\R^d)$ is such that
$
G_\varepsilon(u_\varepsilon)
\leq C
$ {for every $\varepsilon>0$},
then $(u_\varepsilon)$ is compact in $L^1(\Omega;\R^d)$.
\end{item}
\begin{item}[{\rm (ii)}]
The {functionals} $(G_\varepsilon)$ $\Gamma$\hbox{-}converge, as $\varepsilon\to0$, to the functional
\begin{equation*}
G(u)=
\begin{cases}
\displaystyle F(u) + \!\!\int_\Omega\psi(|u|)\,\mathrm{d}x\,, \!\!& \mbox{ if }u\in GSBD^p(\Omega)\cap L^1(\Omega;\R^d)\,,\\
+\infty\,, & \mbox{ otherwise on $L^1(\Omega;\R^d)$, }
\end{cases}
\end{equation*}
with respect to the $L^1$ convergence in $\Omega$.
\end{item}
\end{enumerate}
\end{thm}

\begin{oss}
Existence of minimizers for the functional $G$, and also for $F$ if coupled with a Dirichlet datum, directly follows from Theorem \ref{th: GSDBcompactness} (see \cite{CC18Comp} for details). For fixed $\e$, the functionals $F_\e$ and $G_\e$ are lower semicontinuous in $W^{1,p}(\Omega; \R^d)$, but clearly not coercive. However, as done in \cite[Corollary 3.2]{BraDal97}, one can perturb $f$ with a sequence $(f_\e)$ of functions having linear growth at infinity, and satisfying $f(t)\le f_\e(t)\le f(t)+a_\e t$ for a sequence $a_\e=o(\e)$ as $\e\to0$ and still recover a $\Gamma$-convergence result. 

Their argument would also apply to the present situation: notice that only the $\Gamma$-limsup inequality has to be adapted, and this is straightforward in the space of regular approximating functions provided by Theorem \ref{thm:density}. We omit the details of this generalization. If we now replace $f$ with $f_\e$, existence of minimizers for $G_\e$ in $W^{1,p}$ can be obtained via the direct method. Then, Theorem \ref{thm:mainresult2} (ii) also gives convergence of the minimizers to a minimizer of $G$ in $GSBD^p(\Omega)$.
\end{oss}


\section{Compactness}\label{sec:compactness}

With the following proposition, we prove the compactness statements in Theorem~\ref{thm:mainresult}(i),  and Theorem~\ref{thm:mainresult2} (i), respectively. 
\begin{prop}\label{prop:compactness}
Let $A\subset\Omega$ be any open subset of $\Omega$, and let $F_\varepsilon$, $G_\varepsilon$ be defined as in \eqref{energies0}, and \eqref{energies1}, respectively. Then:
\begin{enumerate}
\begin{item}[{\rm (i)}]
Assume \eqref{eq: alfabeta}, \eqref{eq:coerciv}. If $(u_\varepsilon)\subset L^p(\Omega;\R^d)$ is such that
$
F_\varepsilon(u_\varepsilon, A)
\leq C
$ {for every $\varepsilon>0$},
one can find a sequence $\overline{u}_\varepsilon \in {GSBV^p(A;\R^d)}$ with
\[
\begin{split}
&\overline{u}_\varepsilon-u_\varepsilon \to 0 \mbox{ in measure  on }A \\
&F_\varepsilon(u_\varepsilon, A)\geq c_0 \left(\int_{A} W(\mathcal{E}\overline{u}_\varepsilon)\,\mathrm{d}x + 2 \mathcal{H}^{d-1}(J_{\overline{u}_\varepsilon} \cap A) \right)
\end{split}
\]
{for some $c_0>0$.}
\end{item}
\begin{item}[{\rm (ii)}]
Assume \eqref{eq: alfabeta}, \eqref{eq:coerciv}, and \eqref{eq: hpsi}. If $(u_\varepsilon)\subset L^p(\Omega;\R^d)$ is such that \, \,
$
G_\varepsilon(u_\varepsilon, A)
\leq C
$ {for every $\varepsilon>0$},
then $(u_\varepsilon)$ is compact in $L^1(A;\R^d)$.
\end{item}
\end{enumerate}
\end{prop}

\proof
Let $\varepsilon>0$ and $\delta\in(0,1)$ be fixed \footnote{For the purpose of this proof, one could fix $\delta=\frac12$ from the beginning: however, we prefer to work with {an} arbitrary $\delta$ as the first part of the construction will be used later on.}. It suffices to consider here only the case $f(t)=\min\{at, b\}$ with $a, b>0$. In the general case one can indeed find $a, b>0$ with $f(t)\geq \min\{at, b\}$ for all $t$, using Lemma  \ref{lem: belowapprox}, and deduce the result \emph{a fortiori}.
Hence, let us assume $f(t)=\min\{at, b\}$. We define
\begin{equation}
C_\delta:=\frac{\mathcal{L}^d(B_\varepsilon(0))}{\mathcal{L}^d(B_{(1-\delta)\varepsilon}(0))}=\frac{1}{(1-\delta)^d}
\label{eq:cdelta}
\end{equation}
and the function
\begin{equation*}
\psi_\varepsilon(x):=\varepsilon \dashint_{B_{(1-\delta)\varepsilon}(x)\cap\Omega}W(\mathcal{E} u_\varepsilon(y))\,\mathrm{d}y\,.
\end{equation*}
Correspondingly, we introduce the compact set
\begin{equation}
K_\varepsilon:=\left\{x\in A:\,\, \psi_\varepsilon(x)\geq  C_\delta\, \frac ba\right\}\,.
\label{eq:keps}
\end{equation}
The set $K_\varepsilon$ is actually also depending on the fixed $\delta$ (as well as the sets $K^{\prime \prime}_\varepsilon$ and $K'_\varepsilon$ used below) but we omit this dependence to ease notation. 
We first note that, setting
\begin{equation*}
K^{\prime \prime}_\varepsilon:=\{x\in A:\,\, {\rm dist}(x,K_\varepsilon)\le \delta\epsilon\}\,,
\end{equation*}
then it holds that
\begin{equation}\label{eq: inclusion}
K^{\prime \prime}_\varepsilon\subseteq \left\{x\in A:\,\, \varepsilon \dashint_{B_{\varepsilon}(x)}W(\mathcal{E}u_\varepsilon(y))\,\mathrm{d}y\geq \frac ba\right\}\,.
\end{equation}
Indeed, if $x\in K^{\prime \prime}_\varepsilon$ then $B_\varepsilon(x)\supseteq B_{(1-\delta)\varepsilon}(z)$ for some $z\in K_\varepsilon$, so
\begin{equation*}
\begin{split}
\varepsilon \dashint_{B_{\varepsilon}(x)}W(\mathcal{E} u_\varepsilon(y))\,\mathrm{d}y & \geq \varepsilon \frac{\mathcal{L}^d(B_{(1-\delta)\varepsilon}(z))}{\mathcal{L}^d(B_{\varepsilon}(x))}\dashint_{B_{(1-\delta)\varepsilon}(z)}W(\mathcal{E} u_\varepsilon(y))\,\mathrm{d}y\\
& =\frac{\psi_\varepsilon(z)}{C_\delta}\geq \frac ba\,.
\end{split}
\end{equation*}
Now, from the inclusion \eqref{eq: inclusion} and the fact that $f(t)= b$ for $t\geq \frac ba$, we deduce that
\begin{equation}
\mathcal{L}^d(K^{\prime \prime}_\varepsilon)\leq \frac{\varepsilon}{b} F_\varepsilon(u_\varepsilon, A)\,.
\label{stima1}
\end{equation}
Then, applying the coarea formula to the $1$-Lipschitz function \, $g(x):= \mathrm{dist}(x, K_\epsilon)$ (see for instance \cite[Theorem 3.14]{EvansGariepy92})  in the open set $\{0<g(x)<\delta \epsilon\}\subset K^{\prime \prime}_\varepsilon$ we get
\[
 \frac{\varepsilon}{b} F_\varepsilon(u_\varepsilon, A) \geq \mathcal{L}^d(K^{\prime \prime}_\varepsilon)\ge \int_{0}^ {\delta \epsilon}\mathcal{H}^{d-1}(\{g=t\})\,\mathrm{d}t\,. 
\]
It follows that we can choose $0<\delta^\prime_\e < \delta\e$ such that, for 
\begin{equation}
K'_\varepsilon:=\{x\in A:\,\, {\rm dist}(x,K_\varepsilon)\le \delta^\prime_\e\}\,,
\label{eq:kpeps}
\end{equation}
it holds
\begin{equation}
\mathcal{H}^{d-1}(\partial K'_\varepsilon)=\mathcal{H}^{d-1}(\{x\in A:\,\,{\rm dist}(x,K_\varepsilon)= \delta^\prime_\e\})\leq \frac{1}{\delta b}F_\varepsilon(u_\varepsilon,  A)\,.
\label{eq:boundmennucci}
\end{equation}
For every $\varepsilon>0$, we set
\begin{equation}
\overline{u}_\varepsilon (x)=
\begin{cases}
u_\varepsilon (x)\,, & \mbox{ if }x\in A\backslash K_\varepsilon'\,,\\
0\,, & \mbox{ otherwise. }
\end{cases}
\label{eq:veps}
\end{equation}
Note that from \eqref{stima1} an the bound $F_\varepsilon(u_\varepsilon,A)\leq C$ it follows that
\begin{equation}
\mathcal{L}^d(\{x\in A:\,\, \overline{u}_\varepsilon(x)\neq u_\varepsilon(x)\})\to0\,,
\label{eq:convinmeasure}
\end{equation}
whence $\overline{u}_\varepsilon-u_\varepsilon \to 0$ in measure  on $A$. We prove the following\\
\noindent
{\bf Claim:} there exists a constant $N>0$ depending only on $d$ such that
\begin{equation*}
\varepsilon(1-\delta)^d\dashint_{B_{(1-\delta)\varepsilon}(x)} W(\mathcal{E}\overline{u}_\varepsilon(y))\,\mathrm{d}y\leq N \frac ba
\end{equation*} 
for every $x\in A$.

For this, we first note that by definition of $\overline{u}_\varepsilon$, and  since $W(\mathbf{0})=0$ is the minimum value of $W$, one has $W(\mathcal{E}\overline{u}_\varepsilon(x)) \le W(\mathcal{E}u_\varepsilon(x))$ for a.e. $x$. Now, when $x\in A\backslash K'_\varepsilon$, it holds $x\not\in K_\varepsilon$ so that by definition of $K_\varepsilon$ we have
\begin{equation*}
\varepsilon(1-\delta)^d\dashint_{B_{(1-\delta)\varepsilon}(x)} W(\mathcal{E}\overline{u}_\varepsilon(y))\,\mathrm{d}y\le \varepsilon(1-\delta)^d\dashint_{B_{(1-\delta)\varepsilon}(x)} W(\mathcal{E}u_\varepsilon(y))\,\mathrm{d}y\leq\frac ba\,.
\end{equation*}

On the other hand, if $x\in K'_\varepsilon$, then Remark~\ref{rem:innercovering} shows the existence of a finite subset of $A\backslash K'_\varepsilon$, say $\{x_1,x_2,\dots,x_N\}$, where $N$ only depends on the dimension $d$, such that
\begin{equation}
(A\backslash K_\varepsilon')\cap B_{(1-\delta)\varepsilon}(x) \subseteq \bigcup_{i=1}^N B_{(1-\delta)\varepsilon}(x_i)\,.
\label{eq:covering}
\end{equation}
We then have, with \eqref{eq:veps} and  \eqref{eq:covering},
\begin{equation*}
\begin{split}
\int_{B_{(1-\delta)\varepsilon}(x)} W(\mathcal{E}\overline{u}_\varepsilon(y))\,\mathrm{d}y & = \int_{B_{(1-\delta)\varepsilon}(x)\cap (A\backslash K_\varepsilon')} W(\mathcal{E}u_\varepsilon(y))\,\mathrm{d}y\\
& \leq \sum_{i=1}^N\int_{B_{(1-\delta)\varepsilon}(x_i)} W(\mathcal{E}u_\varepsilon(y))\,\mathrm{d}y\\
& \leq \frac{N \mathcal{L}^d(B_{(1-\delta)\varepsilon})}{\varepsilon(1-\delta)^d}\frac{b}{a}\,,
\end{split}
\end{equation*}
where in the latter inequality we used the fact that the points $x_i\not\in K'_\varepsilon$. This concludes the proof of {the claim}.

Since $W\ge 0$ and $f$ is nondecreasing, we have the  estimate
\begin{equation}\label{eq: elementare}
F_\varepsilon(u_\varepsilon,A)\geq \frac{1}{\varepsilon}\int_Af\left(\varepsilon (1-\delta)^d \dashint_{B_{(1-\delta)\varepsilon}(x)}W(\mathcal{E}u_\varepsilon(y))\,\mathrm{d}y\right)\,\mathrm{d}x\,.
\end{equation} 
Moreover, if $t\le N\frac ba$, one has the elementary inequality $\min\{at, b\}\geq \frac a N t$. With this,  recalling that $W(\mathcal{E}\overline{u}_\varepsilon(x)) \le W(\mathcal{E}u_\varepsilon(x))$ for a.e. $x$, using the Claim, \eqref{eq: elementare} and  the monotonicity of $f$ we obtain the estimate
\begin{equation*}
\begin{split}
F_\varepsilon(u_\varepsilon,A)&\geq \frac{1}{\varepsilon}\int_Af\left(\varepsilon (1-\delta)^d\dashint_{B_{(1-\delta)\varepsilon}(x)}W(\mathcal{E}u_\varepsilon(y))\,\mathrm{d}y\right)\,\mathrm{d}x \\
& \geq \frac{1}{\varepsilon}\int_Af\left(\varepsilon (1-\delta)^d \dashint_{B_{(1-\delta)\varepsilon}(x)}W(\mathcal{E}\overline{u}_\varepsilon(y))\,\mathrm{d}y\right)\,\mathrm{d}x \\
& \geq \frac aN\int_A\left(\dashint_{B_{(1-\delta)\varepsilon}(x)}W(\mathcal{E}\overline{u}_\varepsilon(y))\,\mathrm{d}y\right)\,\mathrm{d}x \\
&=\frac {a}{N\omega_d}\int_{A\times B_1(0)}W(\mathcal{E}\overline{u}_\varepsilon(x+(1-\delta)\e z))\,\mathrm{d}x\,\mathrm{d}z
\end{split}
\end{equation*}
where we changed variables $y=x+(1-\delta)\e z$ and used Fubini's Theorem.  Since $W\ge 0$ with a further change of variables and using \eqref{eq:coerciv} we conclude
\begin{equation}
F_\varepsilon(u_\varepsilon,A)\geq \frac aN\int_{A}W(\mathcal{E}\overline{u}_\varepsilon(x))\,\mathrm{d}x\geq  \frac {ca}N\int_{A}\left|\mathcal{E}\overline{u}_\varepsilon(x)\right|^p\,\mathrm{d}x\,.
\label{eq:bulkbound}
\end{equation}

Since by definition \eqref{eq:veps} we have that $J_{\overline{u}_\varepsilon}=\partial K'_\varepsilon$, with \eqref{eq:boundmennucci}, \eqref{eq:bulkbound} and from the assumption $F_\varepsilon(u_\varepsilon, A)\leq C$ we deduce that ${\overline{u}_\varepsilon \in GSBV^p(A;\R^d)}$. Moreover, setting $c_0:=\frac{1}{2}\min\{\frac{a}{N}, \frac{b\delta}{2}\}$, we infer the lower bound
\begin{equation}
F_\varepsilon(u_\varepsilon, A)\geq c_0 \left(\int_{A} W(\mathcal{E}\overline{u}_\varepsilon)\,\mathrm{d}x + 2 \mathcal{H}^{d-1}(J_{\overline{u}_\varepsilon} \cap A) \right)\,.
\label{eq:bulk+surfbound}
\end{equation} 
Combining with \eqref{eq:convinmeasure}, this proves (i).

For what concerns (ii), notice that by \eqref{eq: hpsi} and \eqref{eq:veps} it holds \,\, $\psi(|\overline{u}_\varepsilon(x)|) \le \psi(|u_\e(x)|)$ for a.e. $x\in \Omega$. Since $F_\varepsilon\leq G_\varepsilon$, from $G_\varepsilon(u_\varepsilon,A)\leq C$ and \eqref{eq:bulk+surfbound} 
we infer that
\begin{equation*}
\int_{A}\psi(|\overline{u}_\varepsilon(x)|)\,\mathrm{d}x + \int_{A}|\mathcal{E}\overline{u}_\varepsilon(x)|^p\,\mathrm{d}x + \mathcal{H}^{d-1}(J_{\overline{u}_\varepsilon}\cap A)\le C<+\infty
\end{equation*}
for all $\e$. Thus, in view of the growth assumption \eqref{eq: hpsi} on $\psi$, by Theorem \ref{th: GSDBcompactness} and Remark \ref{rem: compactness}, the sequence $(\overline{u}_\varepsilon)$ is compact in $L^1(A;\R^d)$.  By \eqref{eq:convinmeasure} and the Vitali dominated convergence Theorem, we conclude that $(u_\varepsilon)$ is compact in $L^1(A;\R^d)$ as well.
\endproof

\section{Estimate from below of the $\Gamma$\hbox{-}limit} \label{sec:estimbelow}

\subsection{Estimate from below of the bulk term}\label{sec:estimbelowbulk}
We begin by giving a first estimate of the $\Gamma$-liminf of the functionals $F_\e$. This estimate is optimal (up to a small error) for the bulk part of the energy, while it is not, for what concerns the surface part. An optimal estimate for this term will be provided separately by means of a slicing argument (see Proposition \ref{prop:lowboundjump} below). As the two parts of the energy are mutually singular, the localization method of Lemma \ref{lem: lemmasup} will eventually allow us to get the $\Gamma$-liminf inequality.

\begin{prop}\label{prop:estimate}
Let $A$ be an open set with $A\subset\subset \Omega$, and consider a sequence $u_\e\in W^{1,p}(\Omega;\R^d)$ converging to $u$ in $L^1(\Omega; \R^d)$. Assume \eqref{eq: alfabeta} and \eqref{eq:coerciv}. Then, for every fixed $0<\delta<1$, there exist a constant $M_\delta$ only depending on $f$ and $\delta$ and a sequence of functions $(v_\varepsilon^\delta)\subset {GSBV^p(A;\R^d)}$ such that
\begin{itemize}
\item[{\rm (i)}] $\displaystyle \alpha(1-\delta)^{2d+1}\int_{A} W(\mathcal{E}v_\varepsilon^\delta(x))\,\mathrm{d}x\leq F_\varepsilon(u_\e,A)$;
\item[{\rm (ii)}] $\displaystyle \mathcal{H}^{d-1}(J_{v_\varepsilon^\delta})\leq M_\delta\, F_\varepsilon(u_\e,A)$;
\item[(iii)] $\displaystyle v_\varepsilon^\delta \to u$ in $L^1(A; \R^d)$ as $\varepsilon \to 0$.
\end{itemize} 
\end{prop}

\proof
We divide the proof {into} two steps.

{\bf Step 1}: we first consider the case $f(t)=\min\{at,b\}$, with $a,b>0$. Observe that in this case the value $\alpha$ given by \eqref{eq: alfabeta} coincides exactly with $a$. We can clearly assume that 
\begin{equation}
\sup_{\varepsilon>0}F_\varepsilon(u_\varepsilon, A)\leq C\,,
\label{eq:equibounded}
\end{equation}
otherwise the assertion is immediate. Corresponding to the fixed $\delta>0$ and for every $\varepsilon>0$, we define the constant $C_\delta$, and the sets $K_\varepsilon$ and $K'_\varepsilon$ as in \eqref{eq:cdelta}, \eqref{eq:keps} and \eqref{eq:kpeps}, respectively. 
We define a sequence $(v_\varepsilon^\delta)$ of functions in ${GSBV^p(A;\R^d)}$ as
\begin{equation}
v_\varepsilon^\delta(x):=
\begin{cases}
\displaystyle\dashint_{B_{(1-\delta)\varepsilon}(x)} u_\varepsilon(y)\,\mathrm{d}y & \mbox{ if }x\in A\backslash K'_\varepsilon\,,\\
0 & \mbox{ otherwise. }
\end{cases}
\label{eq:vdelta}
\end{equation}
Then (iii) immediately follows from Lemma \ref{lem:lbp}(i) and the fact that, by construction and  \eqref{stima1}, it holds $\mathcal L^d(K^\prime_\e)\to 0$ when $\e \to 0$. 
We also have $\mathcal{H}^{d-1}(J_{v_\varepsilon^\delta})\le \mathcal{H}^{d-1}(\partial K'_\varepsilon)$, so that with \eqref{eq:boundmennucci} we deduce (ii) for $M_\delta=\frac1{\delta b}$.

To prove (i), we observe that, since $K_\e \subset K^\prime_\e$ and $A\subset\subset \Omega$, it holds
\[
\varepsilon \dashint_{B_{(1-\delta)\varepsilon}(x)}W(\mathcal{E} u_\varepsilon(y))\,\mathrm{d}y< C_\delta \frac ba
\]
for all $x \in A \setminus K^\prime_\e$. As $C_\delta>1$ and $f(t)=\min\{at,b\}$, we deduce the elementary inequality
\begin{equation}\label{eq: banale}
f\left(\varepsilon \dashint_{B_{(1-\delta)\varepsilon}(x)}W(\mathcal{E} u_\varepsilon(y))\,\mathrm{d}y\right) \geq \frac{a}{C_\delta} \varepsilon \dashint_{B_{(1-\delta)\varepsilon}(x)}W(\mathcal{E} u_\varepsilon(y))\,\mathrm{d}y
\end{equation}
for all $x \in A \setminus K^\prime_\e$. Now, since the function $f$ is concave and $f(0)=0$, 
\begin{equation}
f(\lambda t) \geq \lambda f(t)\,,\quad \forall\, \lambda\in[0,1]\,. 
\label{inequality2}
\end{equation}
With \eqref{eq: banale}, \eqref{inequality2}, the monotonicity of $f$, the convexity of $W$, \eqref{commut} and \eqref{eq:vdelta} we get
\begin{equation*}
\begin{split}
F_\varepsilon(u_\varepsilon,A) & \geq \frac{1}{\varepsilon}\int_{A\backslash K'_\varepsilon}f\left(\varepsilon \dashint_{B_\varepsilon(x)}W(\mathcal{E} u_\varepsilon(y))\,\mathrm{d}y\right)\,\mathrm{d}x \\
  & \geq \frac{1}{\varepsilon C_\delta}\int_{A\backslash K'_\varepsilon}f\left(\varepsilon \dashint_{B_{(1-\delta)\varepsilon}(x)}W(\mathcal{E} u_\varepsilon(y))\,\mathrm{d}y\right)\,\mathrm{d}x \\
& \geq \frac{a}{\varepsilon C_\delta^2}\int_{A\backslash K'_\varepsilon}\left(\varepsilon \dashint_{B_{(1-\delta)\varepsilon}(x)}W(\mathcal{E} u_\varepsilon(y))\,\mathrm{d}y\right)\,\mathrm{d}x\\
& \geq \frac{a}{C_\delta^2}\int_{A\backslash K'_\varepsilon}W\left(\dashint_{B_{(1-\delta)\varepsilon}(x)}\mathcal{E} u_\varepsilon(y)\,\mathrm{d}y\right)\,\mathrm{d}x\\
& = a(1-\delta)^{2d}\int_{A\backslash K'_\varepsilon}W(\mathcal{E}v_\epsilon^\delta(x))\,\mathrm{d}x = a(1-\delta)^{2d}\int_{A}W(\mathcal{E}v_\epsilon^\delta(x))\,\mathrm{d}x\,,
\end{split}
\end{equation*}
which implies assertion (i). This concludes the proof of Step 1.

{\bf Step 2}: for a general $f$ complying with \eqref{eq: alfabeta}, use Lemma \ref{lem: belowapprox} to find $a_\delta, b_\delta>0$ with $a_\delta\ge \alpha(1-\delta)$ and $f(t) \ge \min\{a_\delta t, b_\delta\}$ for all $t \in \R$, and perform the same construction as in the previous step. This gives (iii), (ii) (with $M_\delta:=\frac1{\delta b_\delta}$) and
\[
F_\varepsilon(u_\varepsilon,A)\geq a_\delta(1-\delta)^{2d}\int_{A}W(\mathcal{E}v_\epsilon^\delta(x))\,\mathrm{d}x\geq \alpha(1-\delta)^{2d+1}\int_{A}W(\mathcal{E}v_\epsilon^\delta(x))\,\mathrm{d}x\,,
\]
that is (i).
\endproof

\subsection{Estimate from below of the surface term}\label{sec:estimbelowsurf}

{For any $A\subset\Omega$ open set, we denote by $F'(u, A)$ the lower $\Gamma$\hbox{-}limit of $F_\varepsilon(u,A)$, as defined in \eqref{eq:Gammaliminf}.} 
We note that, since $F_\varepsilon (u,\cdot)$ is superadditive as a set function, the lower $\Gamma$\hbox{-}limit $F'(u,\cdot)$ inherits an analogous property; namely,
\begin{equation}
F'(u, A_1\cup A_2)\geq F'(u, A_1)+F'(u, A_2) \quad \mbox{ whenever }{A}_1\cap {A}_2=\emptyset\,.
\label{eq:superadditive}
\end{equation}

\begin{prop}\label{prop:lowboundjump} 
Assume \eqref{eq: alfabeta} and \eqref{eq:coerciv}. Let $\delta\in(0,1)$ be fixed, and consider a sequence $\e_j \to 0$. Let $A\subset\Omega$ be an open set, $u_j\in W^{1,p}(A;\R^d)$ converging to $u$ in $L^1(A; \R^d)$. Assume that 
\[
\mathop{\lim\inf}_{j\to+\infty}F_{\varepsilon_j}(u_j,A)<+\infty\,.
\]
Then $u \in GSBD^p(A)$ and 
\begin{equation}
\displaystyle\mathop{\lim\inf}_{j\to+\infty}F_{\varepsilon_j}(u_j,A)\geq 2\beta(1-\delta)\int_{J_{u}^\xi\cap A}|\langle \nu,\xi\rangle|\,\mathrm{d}\mathcal{H}^{d-1}
\label{eq:lowboundjump}
\end{equation}
for every $\xi\in S^{d-1}$.
\end{prop}

\proof
It follows from Proposition \ref{prop:compactness} and Theorem \ref{th: GSDBcompactness} that $u \in GSBD^p(A)$.
To prove \eqref{eq:lowboundjump}, we first note that, by virtue of the growth assumption \eqref{eq:coerciv}, we have
\begin{equation*}
W(\mathcal{E}u)\geq c |\mathcal{E}u|^p\geq c |\langle (\mathcal{E}u)\xi,\xi\rangle|^p\,,
\end{equation*}
for every $\xi\in S^{d-1}$. 
Thus, for every fixed $\xi$, since $f$ is non-decreasing, it will be sufficient to provide a lower estimate for the energies
\begin{equation}
F_{\varepsilon_j}^\xi(u_j,A) := \displaystyle\frac{1}{\varepsilon_j}\int_{A}f\left(\frac{c}{\omega_d\varepsilon_j^{d-1}} \int_{B_{{\varepsilon_j}}(x)}|\langle (\mathcal{E}u_j(z))\xi,\xi\rangle|^p\,\mathrm{d}z\right)\,\mathrm{d}x\,.
\label{eq:Gxi}
\end{equation}
We proceed by a slicing argument. If for each $x\in A$ we denote by $x_\xi$ and $x_{\xi^\perp}$ the projections of $x$ onto $\Xi$ and $\Pi^\xi$, respectively, we have
\begin{equation} 
\begin{split}
& F_{\varepsilon_j}^\xi(u_j,A)\\
 & =\displaystyle\int_{\Pi^\xi}\mathrm{d}\mathcal{H}^{d-1}(x_{\xi^\perp})\left(\frac{1}{\varepsilon_j}\int_{A_{\xi,x_{\xi^\perp}}}f\left(\frac{c}{\omega_d\varepsilon_j^{d-1}} \int_{B_{{\varepsilon_j}}(x)}|\langle (\mathcal{E}u_j(z))\xi,\xi\rangle|^p\,\mathrm{d}z\right)\,\mathrm{d}x_\xi\right) \\
 & \geq\displaystyle\int_{\Pi^\xi}\mathrm{d}\mathcal{H}^{d-1}(x_{\xi^\perp})\left(\frac{1}{\varepsilon_j}\int_{A_{\xi,x_{\xi^\perp}}}f\left(\frac{c}{\omega_d\varepsilon_j^{d-1}} \int_{C_{(1-\delta){\varepsilon_j}}^\xi(x)}|\langle (\mathcal{E}u_j(z))\xi,\xi\rangle|^p\,\mathrm{d}z\right)\,\mathrm{d}x_\xi\right)\,,
\end{split}
\label{stimone0}
\end{equation}
since by definition $C_{(1-\delta){\varepsilon_j}}^\xi(x)\subseteq B_{{\varepsilon_j}}(x)$.

We now set
\begin{equation*}
F_{\varepsilon_j}^{\xi,x_{\xi^\perp}}(u_j,A_{\xi,x_{\xi^\perp}}):= \frac{1}{\varepsilon_j}\int_{A_{\xi,x_{\xi^\perp}}}f\left(\frac{c}{\omega_d\varepsilon_j^{d-1}} \int_{C_{(1-\delta){\varepsilon_j}}^\xi(x)}|\langle (\mathcal{E}u_j(z))\xi,\xi\rangle|^p\,\mathrm{d}z\right)\,\mathrm{d}x_\xi\,.
\end{equation*}
For $r_\delta:=\sqrt{\delta(2-\delta)}$, recall that  $C_{(1-\delta){\varepsilon_j}}^\xi(x)=(x_\xi-(1-\delta)\epsilon_j,x_\xi+(1-\delta)\epsilon_j)\times B^{d-1}_{r_\delta\epsilon_j}(x_{\xi^\perp})$, and denote (with a slight abuse of notation) still with $z$ the $(d-1)$-dimensional variable in $B^{d-1}_{r_\delta\epsilon_j}(x_{\xi^\perp})$. Set \begin{equation*}
{w}_j^{\xi,x_{\xi^\perp}}(t):=\dashint_{B^{d-1}_{r_\delta\epsilon_j}(x_{\xi^\perp})}\langle u_j(z+t\xi)),\xi\rangle\,\mathrm{d}z\,.
\end{equation*}
By virtue of Lemma~\ref{lem:lbp}(ii), applied with $\eta_{\varepsilon_j}=r_\delta\varepsilon_j$, we have that ${w}_j^{\xi,x_{\xi^\perp}}$ converges to $u^{\xi,x_{\xi^\perp}}$ in $L^1(A_{\xi,x_{\xi^\perp}})$ for a.e. $x_{\xi^\perp}$. Furthermore, for $c(d,\delta):=\frac{c\omega_{d-1}r_\delta^{d-1}}{\omega_d}$, Fubini's Theorem, Jensen's inequality and the monotonicity of $f$ entail that
\begin{align}\label{stimon}
\begin{split}
&F_{\varepsilon_j}^{\xi,x_{\xi^\perp}}(u_j,A_{\xi,x_{\xi^\perp}})\\
&= \frac{1}{\varepsilon_j}\int_{A_{\xi,x_{\xi^\perp}}}f\left(\frac{c}{\omega_d\varepsilon_j^{d-1}} \int_{B^{d-1}_{r_\delta\epsilon_j}(x_{\xi^\perp})}\mathrm{d}z\int_{x_\xi-(1-\delta)\varepsilon_j}^{x_\xi+(1-\delta)\varepsilon_j}|\langle (\mathcal{E}u_j(z+t\xi))\xi,\xi\rangle|^p\,\mathrm{d}t\right)\,\mathrm{d}x_\xi\\
&= \displaystyle\frac{1}{\varepsilon_j}\int_{A_{\xi,x_{\xi^\perp}}}f\left(\frac{c}{\omega_d\varepsilon_j^{d-1}} \int_{x_\xi-(1-\delta)\varepsilon_j}^{x_\xi+(1-\delta)\varepsilon_j}\left(\int_{B^{d-1}_{r_\delta\epsilon_j}(x_{\xi^\perp})}|\langle (\mathcal{E}u_j(z+t\xi))\xi,\xi\rangle|^p\,\mathrm{d}z\right)\mathrm{d}t\right)\,\mathrm{d}x_\xi \\
& \geq \displaystyle\frac{1}{\varepsilon_j}\int_{A_{\xi,x_{\xi^\perp}}}f\left(\frac{c\omega_{d-1}r_\delta^{d-1}}{\omega_d} \int_{x_\xi-(1-\delta)\varepsilon_j}^{x_\xi+(1-\delta)\varepsilon_j}\left(\dashint_{B^{d-1}_{r_\delta\epsilon_j}(x_{\xi^\perp})}\langle (\mathcal{E}u_j(z+t\xi))\xi,\xi\rangle\,\mathrm{d}z\right)^p\mathrm{d}t\right)\,\mathrm{d}x_\xi\\
&= \displaystyle\frac{1}{\varepsilon_j}\int_{A_{\xi,x_{\xi^\perp}}}f\left({c(d,\delta)}\int_{x_\xi-(1-\delta)\varepsilon_j}^{x_\xi+(1-\delta)\varepsilon_j}|\dot{w}_j^{\xi,x_{\xi^\perp}}(t)|^p\,\mathrm{d}t\right)\,\mathrm{d}x_\xi\\
&= (1-\delta)\displaystyle\frac{1}{(1-\delta)\varepsilon_j}\int_{A_{\xi,x_{\xi^\perp}}}f\left({c(d,\delta)}\int_{x_\xi-(1-\delta)\varepsilon_j}^{x_\xi+(1-\delta)\varepsilon_j}|\dot{w}_j^{\xi,x_{\xi^\perp}}(t)|^p\,\mathrm{d}t\right)\,\mathrm{d}x_\xi\,,
\end{split}
\end{align}

Since the function $t\mapsto f(c(d, \delta)t)$ still tends to $\beta$ when $t\to +\infty$, applying Theorem~\ref{thm:braides} 
to the one-dimensional energies
\begin{equation*}
\widetilde{F}_{\epsilon_j}^{\xi,x_{\xi^\perp}}({w}_j^{\xi,x_{\xi^\perp}}, A_{\xi,x_{\xi^\perp}}):=\displaystyle\frac{1}{(1-\delta)\varepsilon_j}\int_{A_{\xi,x_{\xi^\perp}}}f\left({c(d,\delta)}\int_{x_\xi-(1-\delta)\varepsilon_j}^{x_\xi+(1-\delta)\varepsilon_j}|\dot{w}_j^{\xi,x_{\xi^\perp}}(t)|^p\,\mathrm{d}t\right)\,\mathrm{d}x_\xi
\end{equation*}
we deduce the lower bound
\begin{equation}
\mathop{\lim\inf}_{j\to+\infty} \widetilde{F}_{\epsilon_j}^{\xi,x_{\xi^\perp}}({w}_j^{\xi,x_{\xi^\perp}}, A_{\xi,x_{\xi^\perp}}) \geq  2\beta\#(J_{u^{\xi,x_{\xi^\perp}}}\cap A_{\xi,x_{\xi^\perp}})\,.
\label{stima1D}
\end{equation}
Consequently, from \eqref{stimon} and \eqref{stima1D} we obtain that
\begin{equation*}
\begin{split}
\mathop{\lim\inf}_{j\to+\infty} F_{\varepsilon_j}^{\xi,x_{\xi^\perp}}(u_j, A_{\xi,x_{\xi^\perp}})&\geq (1-\delta) \mathop{\lim\inf}_{j\to+\infty} \widetilde{F}_{\epsilon_j}^{\xi,x_{\xi^\perp}}({w}_j^{\xi,x_{\xi^\perp}}, A_{\xi,x_{\xi^\perp}})\\
& \geq 2\beta(1-\delta)\#(J_{u^{\xi,x_{\xi^\perp}}}\cap A_{\xi,x_{\xi^\perp}})\,.
\end{split}
\end{equation*}
Taking into account \eqref{stimone0}, with Fatou's Lemma we then have
\begin{equation*}
\begin{split}
\displaystyle\mathop{\lim\inf}_{j\to+\infty}F_{\varepsilon_j}(u_j,A)&\geq \mathop{\lim\inf}_{j\to+\infty}\displaystyle\int_{\Pi^\xi}F_{\varepsilon_j}^{\xi,x_{\xi^\perp}}(u_j, A_{\xi,x_{\xi^\perp}})\,\mathrm{d}\mathcal{H}^{d-1}(x_{\xi^\perp}) \\
& \geq \displaystyle\int_{\Pi^\xi} \left(\mathop{\lim\inf}_{j\to+\infty} F_{\varepsilon_j}^{\xi,x_{\xi^\perp}}(u_j, A_{\xi,x_{\xi^\perp}})\right)\, \mathrm{d}\mathcal{H}^{d-1}(x_{\xi^\perp}) \\
& \geq 2\beta(1-\delta)\displaystyle\int_{\Pi^\xi} \#(J_{u^{\xi,x_{\xi^\perp}}}\cap A_{\xi,x_{\xi^\perp}})\, \mathrm{d}\mathcal{H}^{d-1}(x_{\xi^\perp}) \\
&=2\beta(1-\delta)\int_{J_u^\xi\cap A}|\langle \nu_u,\xi\rangle|\,\mathrm{d}\mathcal{H}^{d-1} \,,
\end{split}
\end{equation*}
and the proof of \eqref{eq:lowboundjump} concludes.
\endproof

\subsection{Proof of the $\Gamma$\hbox{-}liminf inequality} \label{sec:gammaliminf}
We summarize the results of the previous sections  in the following Proposition. The $\Gamma$-liminf $G'$ of the sequence $(G_\e)$ is defined as in \eqref{eq:Gammaliminf}, with $G_\e$ in place of $F_\e$. {It holds that $G'(u, A)\geq F'(u, A)$ for each open subset $A\subset \Omega$ and $u \in L^1(A;\R^d)$ (see, e.g., \cite[Proposition~6.7]{DM93}).}
\begin{prop}\label{prop:boundsp}
Assume \eqref{eq: alfabeta}, \eqref{eq:coerciv},  and \eqref{eq: hpsi}. Consider $F_\e$, and $G_\e$ given by \eqref{energies0}, and \eqref{energies1}, respectively.  Let $u\in L^1(\Omega;\R^d)$ and let $A$ be an open subset of $\Omega$, and define $F'(u, A)$ and $G'(u,A)$ by \eqref{eq:Gammaliminf}. If $F'(u,A)<+\infty$, then $u\in GSBD^p(A)$ and
\begin{itemize}
\item[{\rm (i)}] $\displaystyle  F'(u,A)\geq \alpha \int_A W(\mathcal{E}u)\,\mathrm{d}x$\,,
\item[{\rm (ii)}] $\displaystyle G'(u,A)\geq F'(u,A)\geq 2\beta \int_{J_u^\xi\cap A}|\langle \nu_u,\xi\rangle|\,\mathrm{d}\mathcal{H}^{d-1}$
\end{itemize}
for every $\xi\in S^{d-1}$.  If it additionally holds $G'(u, A)<+\infty$, then one also has
\begin{itemize}
\item[{\rm (iii)}] $\displaystyle G'(u,A)\geq  \alpha \int_A W(\mathcal{E}u)\,\mathrm{d}x+ \int_A \psi(|u|)\,\mathrm{d}x$.
\end{itemize}
\end{prop}

\proof
First we note that, by the definition of $\Gamma$\hbox{-}liminf \eqref{eq:Gammaliminf} and a diagonal argument, {there exist subsequences (not relabeled) $(u_j)$ and $(\hat{u}_j)$} converging to $u$ in $L^1(A;\R^d)$ such that
\begin{equation*}
F'(u,A)=\mathop{\lim\inf}_{j\to+\infty} F_{\varepsilon_j}(u_j,A)\,, \quad {G'(u,A)=\mathop{\lim\inf}_{j\to+\infty} G_{\varepsilon_j}(\hat{u}_j,A)}\,.
\end{equation*}
The first equality and Proposition \ref{prop:lowboundjump} give that, if $F'(u,A)<+\infty$, then $u\in GSBD^p(A)$. By the second one, the superadditivity of the liminf, Fatou's lemma and \eqref{eq:Gammaliminf}, we have
\begin{equation*}
\begin{split}
G'(u,A)&=\mathop{\lim\inf}_{j\to+\infty} {G_{\varepsilon_j}(\hat{u}_j,A)\geq \mathop{\lim\inf}_{j\to+\infty} F_{\varepsilon_j}(\hat{u}_j,A)+ \mathop{\lim\inf}_{j\to+\infty} \int_A \psi(|\hat{u}_j|)\,\mathrm{d}x} \\
&\geq F'(u,A)+\int_A \psi(|u|)\,\mathrm{d}x\,.
\end{split}
\end{equation*}
Hence, if (i) is proved, (iii) follows immediately.

We only have to confirm (i) and (ii). To this aim, let $\delta\in(0,1)$ be fixed. Then, by applying Proposition~\ref{prop:estimate} to the sequence $(u_j)$, there exists a sequence of functions $(v_j^\delta)\subset {GSBV^p(A;\R^d)}$, converging to $u$ in $L^1(A)$ as $\varepsilon_j \to 0$, such that
\begin{itemize}
\item[(a)] $\displaystyle(1-\delta)^{2d+1}\int_{A} W(\mathcal{E}v_j^\delta(x))\,\mathrm{d}x\leq F_{\varepsilon_j}(u_j,A)$;
\item[(b)] $\displaystyle \mathcal{H}^{d-1}(J_{v_j^\delta}\cap A)\leq M_\delta  F_{\varepsilon_j}(u_j,A)$.
\end{itemize} 
Combining (a) and (b) with the equiboundedness of $F_{\varepsilon_j}(u_j,A)$, one can apply the lower semicontinuity part of Theorem \ref{th: GSDBcompactness} to the sequence $(v_j^\delta)$. Taking into account that ${A^\infty}=\emptyset$ because $u\in L^1(A; \R^d)$, by the convexity of $W$ and \eqref{eq: GSBD comp}, (ii), we have
\begin{equation*}
\begin{split}
\alpha(1-\delta)^{2d+1}\int_{A} W(\mathcal{E}u(x))\,\mathrm{d}x & \leq \mathop{\lim\inf}_{j\to+\infty}\int_{A} W(\mathcal{E}v_j^\delta(x))\,\mathrm{d}x \\
& \leq \mathop{\lim\inf}_{j\to+\infty}F_{\varepsilon_j}(u_j,A)=F'(u, A)\,.
\end{split}
\end{equation*}
By letting $\delta\to0$ above we then obtain (i)

As for (ii), by Proposition~\ref{prop:lowboundjump}, in particular from \eqref{eq:lowboundjump}, we get
\[
2\beta(1-\delta)\int_{J_{u}^\xi\cap A}|\langle \nu_u,\xi\rangle|\,\mathrm{d}\mathcal{H}^{d-1} \leq \mathop{\lim\inf}_{j\to+\infty} F_{\varepsilon_j}(u_j,A)=F'(u,A)
\]
for every $\xi\in S^{d-1}$, so that (ii) follows by taking the limit as $\delta\to0$ again.
\endproof

We are now in a position to prove the $\Gamma$-liminf inequality.

\begin{prop}\label{prop:lowerbound}
Assume \eqref{eq: alfabeta}, \eqref{eq:coerciv},  and \eqref{eq: hpsi}. Consider $F_\e$, and $G_\e$ given by \eqref{energies0}, and \eqref{energies1}, respectively.  Let $u\in L^1(\Omega;\R^d)$ and let $A$ be an open subset of $\Omega$, and define $F'(u, A)$ and $G'(u,A)$ by \eqref{eq:Gammaliminf}. If $F'(u,A)<+\infty$, then $u\in GSBD^p(A)$ and
\begin{equation*}
\displaystyle F'(u,A)\geq  \alpha\int_A W(\mathcal{E}u)\,\mathrm{d}x + 2\beta \mathcal{H}^{d-1}(J_u\cap A)\,.
\end{equation*}
If it additionally holds $G'(u, A)<+\infty$, then
\begin{equation*}
\displaystyle G'(u,A)\geq  \alpha\int_A W(\mathcal{E}u)\,\mathrm{d}x + 2\beta \mathcal{H}^{d-1}(J_u\cap A)+\int_A \psi(|u|)\,\mathrm{d}x\,.
\end{equation*}
\end{prop}

\proof
We only prove the second inequality, which contains an additional term. Let $(\xi_h)_{h\geq1}$ be a dense sequence in $S^{d-1}$ and let $(\mu_h)_{h\geq0}$ be the sequence of bounded positive measures defined by
\begin{equation*}
\mu_0(A)=\int_{A} \left(\alpha W(\mathcal{E}u(x))+\psi(|u(x)|)\right)\,\mathrm{d}x\,,\,\,\, \mu_h(A)= 2\beta\int_{J_{u}\cap A}\phi^{\xi_h}(x)\,\mathrm{d}\mathcal{H}^{d-1}(x)\,,
\end{equation*}
where
\begin{equation*}
\phi^{\xi_h}(x)=
\begin{cases}
|\langle\nu_u(x),\xi_h\rangle|\,, & \mbox{ if }x\in J_u^{\xi_h}\cap A\,,\\
0\,, & \mbox{ otherwise in }J_u\cap A\,. 
\end{cases}
\end{equation*}
Let $\lambda$ be the bounded positive measure defined by
\begin{equation*}
\lambda(A):= \mathcal{L}^d(A)+\mathcal{H}^{d-1}(J_u\cap A)\,,
\end{equation*} 
and let $(\varphi_h)_{h\geq0}$ be the sequence of $\lambda$-measurable functions on $A$ defined as
\begin{equation*}
\varphi_0(x):=
\begin{cases}
\alpha W(\mathcal{E}u(x)) + \psi(|u(x)|)\,, & \mbox{ if }x\in A\backslash J_u\,,\\
0\,, & \mbox{ if }x\in A\cap J_u\,,
\end{cases}
\end{equation*}
\begin{equation*}
\varphi_h(x):=
\begin{cases}
0\,, & \mbox{ if }x\in A\backslash J_u\,,\\
2\beta\phi^{\xi_h}(x)\,, & \mbox{ if }x\in A\cap J_u\,.
\end{cases}
\end{equation*}
Then $\mu_h(A)=\int_A\varphi_h\mathrm{d}\lambda$ for every $h=0,1,\dots$.

Setting
\begin{equation*}
\varphi(x):=
\begin{cases}
\alpha W(\mathcal{E}u(x)) + \psi(|u(x)|)\,, & \mbox{ if }x\in A\backslash J_u\,,\\
2\beta\,, & \mbox{ if }x\in A\cap J_u\,,
\end{cases}
\end{equation*}
we have that $\sup_{h\geq0}\varphi_h(x)=\varphi(x)$ for $\lambda$-a.e. $x\in A$.

We now define $\mu(A):= G'(u,A)$. By virtue of Proposition~\ref{prop:boundsp} we have that
\begin{equation*}
\mu(A)\geq \mu_h(A)=\int_A\varphi_h\mathrm{d}\lambda
\end{equation*}
for every $h=0,1,\dots$. Since $\mu$ complies with \eqref{eq:superadditive}, as a consequence of Lemma~\ref{lem: lemmasup}, we get
\begin{equation*}
\begin{split}
G'(u,A)=\mu(A)&\geq \int_A\varphi\mathrm{d}\lambda \\
&= \alpha\int_A W(\mathcal{E}u)\,\mathrm{d}x + 2\beta \mathcal{H}^{d-1}(J_u\cap A) + \int_A \psi(|u|)\,\mathrm{d}x\,.
\end{split}
\end{equation*}
\endproof

\section{Estimate from above of the $\Gamma$\hbox{-}limit}\label{sec:upperbound}

{We denote by $F''$ and $G''$ the upper $\Gamma$\hbox{-}limits of $(F_\varepsilon)$ and $(G_\varepsilon)$, respectively, as defined in \eqref{eq:limsupchar}.} 

%

\begin{prop}\label{prop:upperbound}
Let $u\in GSBD^p(\Omega)\cap L^1(\Omega;\R^d)$. Then
\begin{equation}
F''(u)\leq \displaystyle \alpha \int_\Omega W(\mathcal{E}u)\,\mathrm{d}x + 2\beta \mathcal{H}^{d-1}(J_u)\,. 
\label{eq:upperbound}
\end{equation}
If, in addition, it holds that $\int_\Omega \psi(|u|)\,\mathrm{d}x<+\infty$, then 
\begin{equation}
G''(u)\leq \displaystyle \alpha \int_\Omega W(\mathcal{E}u)\,\mathrm{d}x + 2\beta \mathcal{H}^{d-1}(J_u) + \int_\Omega \psi(|u|)\,\mathrm{d}x\,.
\label{eq:upperboundG}
\end{equation}
\end{prop}

\proof
We only prove \eqref{eq:upperbound} by using the density result of Theorem~\ref{thm:density}, as \eqref{eq:upperboundG} follows by an analogous construction with the additional property \eqref{eq:densityfidel}. 

In view of Theorem~\ref{thm:density} and remarks below, by a diagonal argument it is not restrictive to assume that $u\in \mathcal{W}(\Omega;\Rd)$ and that $J_u$ is a closed subset of any of the coordinate hyperplanes, 
that we denote by $K$.

Let  $K_h:=\{x\in\R^d:\, {\rm dist}(x,K)< h\}$ for every $h>0$, and let $\gamma_\epsilon>0$ be a sequence such that $\gamma_\epsilon/\epsilon\to0$ as $\epsilon\to0$. Notice that, for $\varepsilon$ small,
\begin{equation*}
K \subset K_{\gamma_\varepsilon}\subset\subset K_{\gamma_\varepsilon+\varepsilon}\subset\subset \Omega\,,
\end{equation*} 
recalling that $K \subset \Omega$.
Let $\phi_\epsilon$ be a smooth cut-off function between $ K_{\gamma_\varepsilon}$ and $K_{\gamma_\varepsilon+\varepsilon}$, and set
\begin{equation*}
u_\epsilon(x):=u(x)(1-\phi_\epsilon(x))\,.
\end{equation*}
Since $u\in W^{1,\infty}({\Omega}\backslash J_u;\R^d)$ we have $u_\epsilon\in W^{1,\infty}({\Omega};\R^d)$. Note also that, by the Lebesgue Dominated Convergence Theorem, $u_\epsilon\to u$ in $L^1(\Omega;\mathbb{R}^d)$. Moreover, since $u_\epsilon=u$ on $B_\epsilon(x)\cap\Omega$ if $x\not\in K_{\gamma_\varepsilon+\varepsilon}$, we have
\begin{equation}
F_\epsilon(u_\epsilon)\leq \displaystyle\frac{1}{\varepsilon}\int_\Omega f\left(\varepsilon \dashint_{B_\varepsilon(x)\cap\Omega}W(\mathcal{E}u(y))\,\mathrm{d}y\right)\,\mathrm{d}x + \beta \, \frac{\mathcal{L}^d(K_{\gamma_\varepsilon+\varepsilon})}{\epsilon}\,.
\label{stimaint}
\end{equation}
Setting
\begin{equation*}
w_\epsilon(x):= \dashint_{B_\varepsilon(x)\cap\Omega}W(\mathcal{E}u(y))\,\mathrm{d}y\,,
\end{equation*}
we have that $w_\epsilon(x)$ converges to $w(x):=W(\mathcal{E}u(x))$ in $L^1_{\rm loc}(\Omega)$ as $\varepsilon\to0$. Since $f$ complies with \eqref{eq: alfabeta} and it is increasing, there exists $\tilde{\alpha}>\alpha$ such that $f(t)\leq \tilde{\alpha} t$ for every $t\geq0$.
This gives
\begin{equation*}
\frac{1}{\epsilon}f(\epsilon w_\epsilon(x))\leq \tilde{\alpha}w_\epsilon(x) \quad \mbox{ for every }x\in\Omega \mbox{ and every }\epsilon>0\,,
\end{equation*}
and, taking into account that $\displaystyle\lim_{t\to0^+}\frac{f(t)}{t}=\alpha$, we also infer that
\begin{equation*}
\frac{1}{\epsilon}f(\epsilon w_\epsilon(x))\to \alpha w(x) \quad \mbox{ for a.e. }x\in\Omega\,.
\end{equation*}
Thus, by Lebesgue's Dominated Convergence Theorem,
\begin{equation*}
\lim_{\epsilon\to0} \displaystyle\frac{1}{\varepsilon}\int_\Omega f\left(\varepsilon \dashint_{B_\varepsilon(x)\cap\Omega}W(\mathcal{E}u(y))\,\mathrm{d}y\right)\,\mathrm{d}x = \displaystyle \alpha \int_\Omega W(\mathcal{E}u)\,\mathrm{d}x\,.
\end{equation*}
Noting that
\begin{equation*}
\lim_{\epsilon\to0} \beta\, \frac{\mathcal{L}^d(K_{\gamma_\varepsilon+\varepsilon})}{\epsilon} = \lim_{\epsilon\to0} \beta\, \frac{\mathcal{L}^d(K_{\gamma_\varepsilon+\varepsilon})}{2(\gamma_\epsilon+\epsilon)}\frac{2(\gamma_\epsilon+\epsilon)}{\epsilon} = 2\beta\,\mathcal{H}^{d-1}(J_u)\,,
\end{equation*}
from \eqref{stimaint}, the subadditivity of the limsup and \eqref{eq:limsupchar} we get \eqref{eq:upperbound}.
\endproof

\begin{proof}[\it Proof of Theorems \ref{thm:mainresult} and \ref{thm:mainresult2}]
The two results  follow by combining Propositions \ref{prop:compactness}, \ref{prop:lowerbound}, and \ref{prop:upperbound}
\end{proof}

\section*{Acknowledgements}
The authors have been supported by the Italian Ministry of Education, University and Research through the Project “Variational methods for stationary and evolution problems with singularities and interfaces” (PRIN 2017). The authors gratefully acknowledge the anonymous referee for a careful reading of the paper and for her/his interesting remarks leading to improvements of the manuscript.

\bibliographystyle{siam}

\bibliography{references}

\end{document}